\documentclass[11pt,a4paper]{amsart}
\usepackage{amsmath,amsfonts,amsthm,amsopn,color,amssymb,enumitem}
\usepackage[a4paper]{geometry}
\usepackage{palatino}
\usepackage{graphicx}
\usepackage[colorlinks=true]{hyperref}
\hypersetup{urlcolor=blue, citecolor=red, linkcolor=blue}

\usepackage{cite}
\usepackage{relsize}
\usepackage{esint}
\usepackage{verbatim}
\usepackage{mathrsfs}
\usepackage{xcolor}
\usepackage{tikz}

\newcommand{\e}{\varepsilon}

\newcommand{\R}{\mathbb{R}}

\newcommand{\RN}{{\mathbb{R}^N}}

\renewcommand{\l }{\lambda}

\newcommand{\Ua}{{\mathcal{U}_{2}}}
\newcommand{\Uap}{{\mathcal{U}_{2, P_2}}}
\newcommand{\Uam}{{\mathcal{U}_{2, -P_2}}}
\newcommand{\Ub}{{\mathcal{U}_{3}}}
\newcommand{\Ubp}{{\mathcal{U}_{3, P_3}}}
\newcommand{\Ubm}{{\mathcal{U}_{3, -P_3}}}

\newcommand{\Ui}{{\mathcal{U}_{i, \e}}}
\newcommand{\Uae}{{\mathcal{U}_{2, \e}}}
\newcommand{\Ube}{{\mathcal{U}_{3, \e}}}
\newcommand{\Um}{{\mathcal{U}_{i}}}
\newcommand{\Ump}{{\mathcal{U}_{i, P_i}}}
\newcommand{\Umm}{{\mathcal{U}_{i, -P_i}}}

\newcommand{\E}{\mathcal{E}}
\newcommand{\beq}{\begin{equation}}
\newcommand{\eeq}{\end{equation}}

\newtheorem{theorem}{Theorem}[section]
\newtheorem*{theorem*}{Theorem}
\newtheorem{lemma}[theorem]{Lemma}

\newtheorem{proposition}[theorem]{Proposition}

\theoremstyle{definition}
\newtheorem{remark}[theorem]{Remark}

\title[Lotka-Volterra type interactions ]{Partially concentrating solutions for systems with Lotka-Volterra type interactions }

\author{Sabrina Caputo}
\address[S. Caputo]{Dipartimento di Matematica, Universit\`a degli studi di Bari ``Aldo Moro'', via Edoardo Orabona 4,70125 Bari, Italy}
\email{s.caputo28@studenti.uniba.it}

\author{Giusi Vaira}
\address[G.Vaira]
{Dipartimento di Matematica, Universit\`a degli studi di Bari ``Aldo Moro'', via Edoardo Orabona 4,70125 Bari, Italy}
\email{giusi.vaira@uniba.it}

\date{}


\thanks{The last author is partially supported by MUR-PRIN-2022 PNRR ``Linear and Nonlinear PDE's: New directions and Applications'', P2022YFAJH and by the INdAM-GNAMPA group.}
\begin{document}
\maketitle

\begin{abstract}
In this paper we consider the existence of standing waves for a coupled system of $k$  equations with Lotka-Volterra type interaction. We prove the existence of a standing wave solution with all nontrivial components satisfying a prescribed asymptotic profile. In particular, the $k-1$-last components of such solution exhibits a concentrating behavior, while the first one keeps a quantum nature. We analyze first in detail the result with three equations since this is the first case in which the coupling has a role contrary to what happens when only two densities appear. We also discuss the existence of solutions of this form for systems with other kind of couplings making a comparison with Lotka-Volterra type systems. 
\end{abstract}
\section{Introduction}
Several physical phenomena can be described by a certain number of densities (of mass, population, probability, ...) distributed in a domain or on all the space and subject to laws of diffusion, reaction and competitive interaction. Whenever the competitive interaction is the prevailing phenomenon, the several densities cannot coexist and tend to segregate, hence determining a partition of the domain ({\it Gause's experimental principle of competitive exclusion (1932)}).\\ As a model problem, we will consider a system of stationary equations 
with Lotka-Volterra type interaction. This is of particular interest since these systems are the most popular mathematical models for the dynamics of many populations subject to spatial diffusion, internal reaction and either cooperative or competitive interactions.\\
Indeed, such models are related with reaction-diffusion systems where the reaction is the sum of an intra-specific term, often expressed by logistic type functions, and an inter-specific interaction one, usually quadratic. The study of this reaction-diffusion system has a long history and there exists a large literature on the subject. However, most of these works are concerned with the case of two species. As far as we know, the study in the case of many competing species has been much more limited, starting from two pioneering papers by Dancer and Du \cite{DD1, DD2} in the 1990s, where the competition of three species were considered.\\ In what follows we will consider a system of $k-$ non-negative densities in all $\mathbb R^N$ subject to diffusion, reaction and competitive or cooperative interaction in the stationary case of the form 
\begin{equation}\label{SYS}
-d_i\Delta u_i=f_i(u_i)+\beta u_i\sum_{j\neq i}a_{ij}u_j\,\,\,i=2,\ldots, k\quad\mbox{in}\,\, \mathbb R^N.
\end{equation}
Here the matrix $(a_{ij})_{ij}$ is the matrix of the interspecific competition coefficients and the parameter $\beta$  measures the strength of the (competitive or cooperative) interaction. Accordingly, we will distinguish between the symmetric (i.e. when $a_{ij} = a_{ji}$, for every $i, j$) and the asymmetric case (i.e. when $a_{ij} \neq a_{ji}$, for some $i, j$). For concreteness we require the reaction terms $f_i$ are of specific form with $f_i(0) = 0$. The coefficient $d_i>0$ represents the diffusion coefficients.\\
Many papers deal with system \eqref{SYS} in the strong competition regime, that is when the parameter $\beta$ diverges to $-\infty$. In this case, the components satisfy uniform bounds in H\"older norms and converge, up to subsequences, to some limit profiles, having disjoint supports: {\it the segregated states}.\\
In the last decade, both the asymptotics and the qualitative properties of the limit segregated profiles have been the object of an intensive study, mostly in the symmetric case, by different teams \cite{CTV, CTV1, CTV2, CKL, SZ, WZ} and recently in the asymmetric case in \cite{TVZ} analyzing the behavior of a new kind of solutions known as spiralling solutions. \\
On the contrary, when one has to show existence of such systems the situation is more complicated since the problem is not variational and the standard technique cannot be used.\\
In this paper we use a different approach in order to establish an existence result for problems like \eqref{SYS}.\\ We inspire to the recent work \cite{PPVV} where a system of two equations but with an interaction (symmetric) of Gross-Pitaevskii type is considered.\\
Indeed, we will consider the case of three equations (in Section \ref{generalizzazione} we discuss the case of an arbitrary numbers of equations) of the following form 
\begin{equation}\label{DS}
\left\{\begin{aligned}-\Delta u_1+V(x)u_1&=\mu_1 u_1^3+u_1\left(a_{12}u_2+a_{13}u_3\right),\quad&\mbox{in}\,\, \mathbb R^N\\
-\e^2\Delta u_2+W_2(x)u_2 &=\mu_2 u_2^3+ u_2\left(a_{21}u_1+\beta a_{23}u_3\right)\,\quad&\mbox{in}\,\, \mathbb R^N\\
-\e^2\Delta u_3+W_3(x)u_3 &=\mu_3 u_3^3+ u_3\left(a_{31}u_1+\beta a_{32}u_2\right)\,\quad&\mbox{in}\,\, \mathbb R^N\end{aligned}\right.
\end{equation}
where  $N=2, 3$ and $a_{ij}\neq a_{ji}$. We assume that $\beta>0$. We remark that $\beta$ multiplies only the interaction between the second and the third component.\\ Hence our study deals with a system with small diffusion (i.e. in the second and in the third equations $d_i=\e^2>0$ while $d_1=1$) and, moreover, $\mu_i>0$ and $a_{21},a_{31}<0$ (this means that the first and the second densities and the first and the third densities repel each others).

In what follows the class of potentials that we consider satisfies the following assumptions.
\begin{itemize}
\item[$\bf{(V_{1})}$] $V(x)=V(|x|)$ is a radially symmetric function satisfying
\begin{equation} \label{eq:V0}
V \in C^{0}(\mathbb R^N) \ \hbox{and}\  \inf_{\R^{N}}V(x)>0.
\end{equation}
Moreover, $V$ is  such that  there  exists $ \Upsilon \in C^{3}(\R^{N})\cap H^{2}(\R^{N})$  unique positive  radial solution to
\begin{equation} \label{eq:V1}
\begin{cases}
-\Delta \Upsilon+V(x)\Upsilon=\mu_1 \Upsilon^{3}
\\
\Upsilon(x)\to 0 \text{ as } |x|\to +\infty,
\end{cases}
\end{equation}
which is non-degenerate in the space $H^{1}_{e}(\R^{N})$ of functions even with respect to each variables, where
\begin{equation}\label{he}
H^{1}_{e}(\R^{N})\!:=\!\left\{u\in H^{1}(\R^N) : u (\!x_{1},\dots, x_{i},\dots x_{N}\!)\!=\!u (\!x_{1}, \dots,-x_{i}, \dots, x_{N}\!), \,i=1,...\,,N\!
\right\}.
\end{equation}
This means that the only solutions to
$$\left\{\begin{aligned}
&- \Delta z +V(x)z =3\mu_1 \Upsilon^2 z\ \hbox{in}\ \mathbb R^N\\
&z\in H^{1}_{e}(\R^{N})\\
\end{aligned}\right.$$
is the trivial one.\\

\item[$\bf{(V_{2})}$] 
The potential $V$ is such that for every $f\in L^m(\R^N)$, $2\leq m<+\infty$, there exists a unique solution $u\in W^{2, m}(\mathbb R^N)$ of the equation
\[
-\Delta u+V(x)u=f,
\]
and
\[
\|u\|_{W^{2, m}(\R^{N})}\leq C\|f\|_{L^{m}(\R^{N})}.
\]

\item[$\bf{(W_i)}$]  $W_i(x)$ are even with respect to all the variables, 
\begin{equation} \label{eq:W0}
W_i \in C^{3}(\R^{N})\ \hbox{and}\ \inf_{\R^{N}}  W_i(x)> 0\,\, \mbox{for any}\,\, i=2, 3.
\end{equation}
\end{itemize}
The main result that we obtain is the following.

\begin{theorem}\label{thm1}
Let $N=2, 3$ and suppose that $\bf{(V_1)}$, $\bf{(V_2)}$ and $\bf{(W_i)}$ hold. Let $$\omega_i(x)=W_i(x)-a_{i1}\Upsilon (x)$$ and set $\omega_i:=\omega_i(0)>0$. Assume that
\begin{itemize}
\item if $\omega_2 \leq \frac{\omega_3}{4}$ we let  $$\partial_{11}^2\omega_2(0)<0$$ and $$\partial_{22}^2\omega_3(0)<0\quad\mbox{and}\quad a_{32}>0$$
$$\partial_{22}^2\omega_3(0)>0\quad\mbox{and}\quad a_{32}<0$$
\item if $\omega_3 \leq \frac{\omega_2}{4}$ we let  $$\partial_{22}^2\omega_3(0)<0$$ and $$\partial_{11}^2\omega_2(0)<0\quad\mbox{and}\quad a_{23}>0$$
$$\partial_{11}^2\omega_2(0)>0\quad\mbox{and}\quad a_{23}<0$$
\item if $\frac{\omega_3}{4}<\omega_2 < 4 \omega_3$ and $\omega_2\neq \omega_3$, we let  $$\partial_{11}^2\omega_2(0)<0\,\, \mbox{and}\,\, \partial_{22}^2\omega_3(0)<0.$$
\item if $\omega_2=\omega_3=\omega_0$ we let $$\partial_{22}^2\omega_3(0)<0\quad\mbox{and}\quad a_{32}>0$$
$$\partial_{22}^2\omega_3(0)>0\quad\mbox{and}\quad a_{32}<0$$
and
$$\partial_{11}^2\omega_2(0)<0\quad\mbox{and}\quad a_{23}>0$$
$$\partial_{11}^2\omega_2(0)>0\quad\mbox{and}\quad a_{23}<0$$
\end{itemize}
Then there exists $\e_0>0$ such that for any $\e\in(0, \e_0)$ there exists $\beta_\e>0$ such that for any $\beta\in (0, \beta_\e)$ the system \eqref{DS} has a positive solution $(u_{1, \e}, u_{2, \e}, u_{3, \e})$ which is even with respect to each variable and having the following asymptotic profile as $\e\to 0$ $$u_{1, \e}(x)\sim \Upsilon (x), \quad u_{i, \e}(x)\sim \Um\left(\frac{x-P_i}{\e}\right)+\Um\left(\frac{x+P_i}{\e}\right),\,\, i=2,3$$ where $\Upsilon $ solves \eqref{eq:V1} and $\Um$ is the positive radial solution of 
\begin{equation}\label{limi}
-\Delta \Um+\omega_i \Um=\mu_i (\Um)^3\, \quad\mbox{in}\,\, \mathbb R^N.
\end{equation}
The peaks $P_i$ and $-P_i$ collapse to the origin as $$\pm P_2=\rho_{2, \e}(\pm 1, 0, 0)\, \mbox{with}\,\, \frac{\rho_{2, \e}}{\e|\ln\e|}\to \mathfrak a \, \mbox{as}\, \e\to 0$$ and $$\pm P_3=\rho_{3, \e}(0, \pm 1, 0)\, \mbox{with}\,\, \frac{\rho_{3, \e}}{\e|\ln\e|}\to \mathfrak b \, \mbox{as}\, \e\to 0$$ where $\mathfrak a$ and $\mathfrak b$ suitably depend on $\omega_i$.
\end{theorem}
Theorem \ref{thm1} states the existence of a solution whose first component is a genuine solution to \eqref{eq:V1}, in particular it does not concentrate, and the second and the third components concentrates at two opposite points that collapse to the origin as $\e\to 0$.\\
The first component plays the role of an additional potential in the second and in third equation that are singularly perturbed. This is why we need to introduce a modified potential $\omega_i(x)$. \\ Here we consider the case in which $a_{i1}<0$ for $i=2, 3$ (i.e. repulsive regime) while we can consider $a_{ij}$ with $j\neq 1$ and $i=2, 3$ positive or negative depending on the properties of the modified potentials $\omega_i(x)$ and $\beta$, which describes the strength of the interaction, is sufficiently small. So the interaction between the last two densities is a weak repulsive or attractive interaction. The Theorem is stated in dimensions two and three. For the dimension one, it is needed to correct the ansatz again. Hence, in order to make the paper reasonable, we do not consider this case.
\\
When only two equations are considered in \eqref{DS} then the result that we obtain is the following.
\begin{theorem}\label{thm2}
Let $N=2, 3$ and suppose that $\bf{(V_1)}$, $\bf{(V_2)}$ and $\bf{(W_2)}$ hold. Let $$\omega_2(x)=W_2(x)-a_{21}\Upsilon (x)$$ and set $\omega_2:=\omega_2(0)>0$. Assume that 
  $$\partial_{11}^2\omega_2(0)<0.$$ Then there exists $\e_0>0$ such that for any $\e\in(0, \e_0)$ the system \eqref{DS} with two equations has a positive solution $(u_{1, \e}, u_{2, \e})$ which is even with respect to each variable and having the following asymptotic profile as $\e\to 0$ $$u_{1, \e}(x)\sim \Upsilon (x), \quad u_{2, \e}(x)\sim \Ua\left(\frac{x-P_2}{\e}\right)+\Um\left(\frac{x+P_2}{\e}\right),\,\, $$ where $\Upsilon $ solves \eqref{eq:V1} and $\Ua$ is the positive radial solution of 
\begin{equation}\label{limi}
-\Delta \Ua+\omega_2 \Ua=\mu_2 (\Ua)^3\, \quad\mbox{in}\,\, \mathbb R^N.
\end{equation}
The peaks $P_2$ and $-P_2$ collapse to the origin as $$\pm P_2=\rho_{2, \e}(\pm 1, 0, 0)\, \mbox{with}\,\, \frac{\rho_{2, \e}}{\e|\ln\e|}\to \mathfrak a \, \mbox{as}\, \e\to 0$$ where $\mathfrak a$ suitably depend on $\omega_2$.
\end{theorem}
\noindent The construction in Theorem \ref{thm2} is as in Theorem \ref{thm1}. However, since only two densities appear, we do not have the terms with $\beta$.\\ 
In the last Section we analyze the case of Gross-Pitaevskii type interaction and we use the configuration of points for the Lotka-Volterra system in order to generalize the result in \cite{PPVV} where only two equations are considered. \\ Here the form of the interaction permit to find a solution of the same form as in Theorem \ref{thm1} but without assuming the smallness of the interactions (see Theorem \ref{thm2}).\\ Indeed, we have the following result.
\begin{theorem}\label{thm3}
Let $N=2, 3$ and suppose that $\bf{(V_1)}$, $\bf{(V_2)}$ and $\bf{(W_i)}$ hold. Let $$\omega_i(x)=W_i(x)-a_{i1}\Upsilon (x)$$ and set $\omega_i:=\omega_i(0)\equiv \omega_0>0$. Assume that
 $$\partial_{11}^2\omega_2(0)<0\, \mbox{and}\,\partial_{22}^2\omega_3(0)<0.$$
There there exists $\e_0>0$ such that for any $\e\in(0, \e_0)$  the system 
$$ \left\{
\begin{aligned}
-  \Delta u_1 +V (x)u_1&=\mu_1 u_1^3+ u_1 \sum\limits_{j=2}^k \beta_{1j} u_j^2 &
 \hbox{in}\ \mathbb R^N,
 \\
-\e^2 \Delta u_j +W_j(x)u_j&=\mu_j u_j^3+u_j  \left(\beta_{j1} u_1^2+ \sum\limits_{i\not=j} \beta_{ji} u_i^2\right) & \hbox{in}\ \mathbb R^N,\ j=2,3.
\end{aligned}\right.
$$
has a positive solution $(u_{1, \e}, u_{2, \e}, u_{3, \e})$ which even with respect to each variable and having the following asymptotic profile as $\e\to 0$ $$u_{1, \e}(x)\sim \Upsilon (x), \quad u_{i, \e}(x)\sim \Um\left(\frac{x-P_i}{\e}\right)+\Um\left(\frac{x+P_i}{\e}\right),\,\, i=1, 2$$ where $\Upsilon $ solves \eqref{eq:V1} and $\Um$ is the positive radial solution of 
\begin{equation}\label{limi}
-\Delta \Um+\omega_0 \Um=\mu_i (\Um)^3\, \quad\mbox{in}\,\, \mathbb R^N.
\end{equation}
The peaks $P_i$ and $-P_i$ collapse to the origin as $$\pm P_2=\rho_{2, \e}(\pm 1, 0, 0)\, \mbox{with}\,\, \frac{\rho_{2, \e}}{\e|\ln\e|}\to \mathfrak a \, \mbox{as}\, \e\to 0$$ and $$\pm P_3=\rho_{3, \e}(0, \pm 1, 0)\, \mbox{with}\,\, \frac{\rho_{3, \e}}{\e|\ln\e|}\to \mathfrak b \, \mbox{as}\, \e\to 0$$ where $\mathfrak a$ and $\mathfrak b$ suitably depend on $\omega_0$.
\end{theorem}
We remark that in Theorem \ref{thm3} we consider only the case in which all the $\omega_i$ coincide with some $\omega_0$ in order to make a comparison with Lotka-Volterra type interactions.\\
At the end we remark a possible configuration for considering $k\geq 3$ equations. However the computations are not so easy, so we do not consider the fully general case.\\
We finally remark that in Lotka-Volterra type system the interaction plays a role in order to obtain a solution, while in Gross-Pitaevskii type system the interaction does not play any role even with more equations.\\

The paper is organized in the following way: in Section \ref{setting} we introduce the good spaces and the ansatz of the solution we look for and its properties. In Section \ref{ausiliaria} we find the error term and in Section \ref{teorema} we prove Theorem \ref{thm1}. In Section \ref{generalizzazione} we consider other kind of systems. Finally, in Section \ref{appendix} we recall some useful lemmas.

\section{Setting of the problem}\label{setting}
We let$$H^2_V:=\left\{u\in H^2(\mathbb R^N)\,:\, \int_{\mathbb R^N} V(x) u^2\, dx <+\infty\right\},$$
$$H^2_{W_{i, \e}}:=\left\{u\in H^2(\mathbb R^N)\, :\, \int_{\mathbb R^N} W_i(\e x)u^2\, dx <+\infty\right\}$$
equipped with the norms
$$\|u\|_V:=\left(\int_{\mathbb R^N}\sum_{|\alpha|=2}|D^\alpha u|^2+\int_{\mathbb R^N}|\nabla u|^2+\int_{\mathbb R^N}V(x) u^2\right)^{\frac 12}$$ and 
$$\|u\|_{i, \e}:=\left(\int_{\mathbb R^N}\sum_{|\alpha|=2}|D^\alpha u|^2+\int_{\mathbb R^N}|\nabla u|^2+\int_{\mathbb R^N}W_i(\e x)u^2\right)^{\frac 12}.$$
Performing a change of variables, we are lead to look for a solution $(u_1, u_2, u_3)\in X $ of
\begin{equation}\label{DS1}
\left\{\begin{aligned}-\Delta u_1+V(x)u_1&=\mu_1 u_1^3+a_{12}u_1u_2\left(\frac x \e\right)+a_{13}u_1u_3\left(\frac x \e\right),\quad&\mbox{in}\,\, \mathbb R^N\\
-\Delta u_2+W_2(\e x)u_2 &=\mu_2 u_2^3+a_{21}u_2u_1(\e x)+\beta a_{23}u_2u_3(x)\quad&\mbox{in}\,\, \mathbb R^N\\
-\Delta u_3+W_3(\e x)u_3 &=\mu_3 u_3^3+a_{31}u_3u_1(\e x)+\beta a_{32}u_2u_3(x)\quad&\mbox{in}\,\, \mathbb R^N\\
\end{aligned}\right.
\end{equation}
where \begin{equation}\label{X}X:=\left\{(u_1, u_2, u_3)\in H^2_V\times H^2_{W_{1, \e}}\times H^2_{W_{3, \e}}\, :\, u_i\,\, \hbox{are even functions}\right\}.\end{equation}
\subsection{The ansatz and the correction terms}
Since $\bf{(V_{1})}$ and $\bf{(V_{2})}$ hold, then we can consider $\Upsilon$ the solution of 
\begin{equation}\label{lim1}
-\Delta\Upsilon+V(x)\Upsilon=\mu_1\Upsilon^3\, \quad\mbox{in}\,\, \mathbb R^N
\end{equation}
and $\Um$ be the solution of 
\begin{equation}\label{limi}
-\Delta \Um+\omega_i \Um=\mu_i (\Um)^3\, \quad\mbox{in}\,\, \mathbb R^N
\end{equation}
where $$\omega_i:=W_i(0)-a_{i1}\Upsilon(0).$$ 
 We also let $$\omega_i(\e x)=W_i(\e x)-a_{i1}\Upsilon (\e x)\,\, i=2, 3.$$
Since $a_{i1}<0$ then $\omega_i>0$.\\
We look for a solution $(u_1, u_2, u_3 )$ of \eqref{DS1} of the form
$$u_1(x)=\Upsilon(x)+\varphi_1(x)+\Phi_1(x),\quad u_i(x)= \Ui(x)+\varphi_i(x)+\Phi_i(x), i=2, 3$$
where
\begin{equation}\label{hatU}
 \Ui(x)=\Um\left(x-\frac{P_i}{\e}\right)+\Um\left(x+\frac{P_i}{\e}\right):=\Umm+\Ump
\end{equation}
and the concentration points satisfy 
\begin{equation}\label{P1}
P_2=\rho_{2, \e} P_0=\rho_{2, \e}(1, 0, \ldots, 0),\,\, \frac{\rho_{2, \e}}{\e}\to+\infty\, \mbox{as}\,\, \e\to0.
\end{equation}
\begin{equation}\label{P2}
P_3=\rho_{3, \e} \bar P_0=\rho_{3, \e}(0, 1, \ldots, 0),\,\, \frac{\rho_{3, \e}}{\e}\to+\infty\, \mbox{as}\,\, \e\to0.
\end{equation}
Moreover $$\frac{|P_2\pm P_3|}{\e}=\frac{\sqrt{\rho_{2, \e}^2+\rho_{3, \e}^2}}{\e}\to +\infty\, \mbox{as}\,\, \e\to0.
$$
Finally $
\rho_{i, \e} \in \mathcal D_{i, \e}^j
$
where $\mathcal D_{i, \e}^j$ are defined in this way: let $b>0$
\begin{enumerate}
\item if $\omega_2\leq\frac{\omega_3}{4}$ then $$\mathcal D_{2, \e}^1:=\left[\left(\frac{1}{\sqrt{\omega_2}}-\delta\right)\e|\ln\e|, \left(\frac{1}{\sqrt{\omega_2}}+\delta\right)\e|\ln\e|\right]$$ and $$\mathcal D_{3, \e}^1:=\left[\left(\frac{\sqrt {(1-b)(3-b)}}{\sqrt{\omega_2}}-\delta\right)\e|\ln\e|, \left(\frac{\sqrt {(1-b)(3-b)}}{\sqrt{\omega_2}}+\delta\right)\e|\ln\e|\right]$$ with $b<1$ and $\delta>0$ small.\\
\item if $\omega_3\leq\frac{\omega_2}{4}$ then $$\mathcal D_{2, \e}^2:=\left[\left(\frac{\sqrt {(1-b)(3-b)}}{\sqrt{\omega_3}}-\delta\right)\e|\ln\e|, \left(\frac{\sqrt {(1-b)(3-b)}}{\sqrt{\omega_3}}+\delta\right)\e|\ln\e|\right]$$ and $$\mathcal D_{3, \e}^2:=\left[\left(\frac{1}{\sqrt{\omega_3}}-\delta\right)\e|\ln\e|, \left(\frac{1}{\sqrt{\omega_3}}+\delta\right)\e|\ln\e|\right]$$ with $b<1$ and $\delta>0$ small.\\
\item if $\frac{\omega_3}{4}<\omega_2<4 \omega_3$ and $\omega_2\neq \omega_3$ then $$\mathcal D_{2, \e}^3:=\left[\left(\frac{1}{\sqrt{\omega_2}}-\delta,\frac{1}{\sqrt{\omega_2}} +\delta\right)\e|\ln\e|\right]$$ and $$\mathcal D_{3, \e}^3:=\left[\left(\frac{1}{\sqrt{\omega_3}}-\delta, \frac{1}{\sqrt{\omega_3}}+\delta\right)\e|\ln\e|\right]$$ with $\delta>0$ small.\\
\item if $\omega_2=\omega_3=\omega_0$ we let
$$\mathcal D_{2, \e}^4:=\left[\left(\frac{2-b}{\sqrt{2\omega_0}}-\delta,\frac{2-b}{\sqrt{2\omega_0}} +\delta\right)\e|\ln\e|\right]$$ and $$\mathcal D_{3, \e}^4:=\left[\left(\frac{2-b}{\sqrt{2\omega_0}}-\delta, \frac{2-b}{\sqrt{2\omega_0}}+\delta\right)\e|\ln\e|\right]$$ with $b<2$ and $\delta>0$ small.\\
\end{enumerate}
We also define $$\mathcal Z_{\e}:=\e \frac{\partial \Uae}{\partial \rho_{2, \e}}=\partial_1 \Ua\left(x+\frac{P_2}{\e}\right)-\partial_1 \Ua\left(x-\frac{P_2}{\e}\right)$$ and $$\mathcal Y_{\e}:=\e \frac{\partial \Ube}{\partial \rho_{2, \e}}=\partial_2 \Ub\left(x+\frac{P_3}{\e}\right)-\partial_2 \Ub\left(x-\frac{P_3}{\e}\right).$$ By simple computations we find that $\mathcal Z_\e$ solves the linear problem
\beq\label{Ze}
-\Delta \mathcal Z_\e +\omega_2\mathcal Z_\e =3\mu_2 \left((\Uap)^2\partial_1 \Uap-(\Uam)^2\partial_1 \Uam\right)\eeq
and 
$\mathcal Y_\e$ solves the linear problem
\beq\label{Ye}
-\Delta \mathcal Y_\e +\omega_3\mathcal Y_\e =3\mu_3 \left((\Ubp)^2\partial_2 \Ubp-(\Ubm)^2\partial_2 \Ubm\right).\eeq
We remark that $\Ui$, $\mathcal Z_\e$ and $\mathcal Y_\e$ are even for $i=2, 3$.\\\\

\noindent The functions $\Phi_i$ are remainder terms that satisfies suitable orthogonality conditions. Indeed, if we let $$\mathcal K^\bot:=\left\{(\psi_1, \psi_2, \psi_3)\in X\, :\, \int_{\mathbb R^N}\psi_2 \mathcal Z_\e=0, \,\, \int_{\mathbb R^N}\psi_3\mathcal Y_\e=0\right\}$$ then we look for $(\Phi_1, \Phi_2, \Phi_3)\in\mathcal K^\bot$.\\
\begin{remark}\label{rema1}
It is useful to recall the classical results concerning the case of  
 constant potential, i.e.
 \begin{equation}\label{Uc}
-\Delta U+\lambda U=\nu U^{3}\ \hbox{in}\ \R^N,\ U\in H^1(\R^N).
\end{equation}
It is well known that \eqref{Uc} has  an unique positive solution which is radially symmetric   and also that
the  set of solution of the corresponding linearized 
equation
$$-\Delta z+\lambda z=3\nu U^{2}z\ \hbox{in}\ \R^N,\ z\in H^1(\R^N)$$
 is spanned by the $N$ partial derivatives  $ {\partial U\over \partial x_i}$ which are  odd in each variable.
In addition,   $U$    is radially decreasing and it satisfies the following exponential decay (see \cite{beli1,beli2, kwong})
\begin{equation}\label{eq:Udecay}
\lim_{|x|\to\infty}U(x)e^{\sqrt{\lambda}|x|}|x|^{\frac{N-1}2}=C_{0}>0,
\qquad
\lim_{|x|\to\infty}\frac{U'(x)}{U(x)}=-1.
\end{equation}\\
 \end{remark}
\noindent In what follows we need to understand the correction terms and their properties in the following lemmas.

\begin{lemma}\label{hatvarphi}
For $i=2, 3$, there exists a unique even $\eta_i\in H^2_V(\mathbb R^N)$ solution of the equation
\begin{equation}\label{equazionehatvarphi}
-\Delta\eta_i+\left(V(x)-3\mu_1\Upsilon^2(x)\right)\eta_i=\Upsilon(x) \Ui\left(\frac x \e\right).
\end{equation}
Moreover $\|\eta_i\|_{W^{2, m}(\mathbb R^N)}\lesssim \e^{\frac N m}$ and $\|\eta_i\|_{C^{1, 1-\frac N m}(\mathbb R^N)}\lesssim \e^{\frac N m}$ for any $m\geq 2$.
\end{lemma}
\begin{proof}
We can reason as in Lemma 2.1 of \cite{PPVV} to get the conclusion.
\end{proof}
 We define the even function $$\varphi_1(x):=a_{12}\eta_2(x)+a_{13}\eta_3(x).$$  \begin{remark}\label{stimaphi1}
Since $\nabla\eta_i(0)=0$ we deduce that for any $m\geq 2$
$$
|\eta_i(y)-\eta_i(0)|\lesssim |y|^{2-\frac N m}\|\eta_i\|_{C^{1, 1-\frac N m}(\mathbb R^N)}\lesssim \e^{\frac N m}|y|^{2-\frac N m}.
$$
Hence we also obtain that 
\beq\label{stima1}
|\varphi_1(y)-\varphi_1(0)|\lesssim |y|^{2-\frac N m}\left(\|\eta_2\|_{C^{1, 1-\frac N m}}+\|\eta_3\|_{C^{1, 1-\frac N m}(\mathbb R^N)}\right)\lesssim \e^{\frac N m}|y|^{2-\frac N m}\eeq and 
$$\|\varphi_1\|_{W^{2, m}(\mathbb R^N)}\lesssim \e^{\frac N m},\quad \|\varphi_1\|_{C^{1, 1-\frac N m}(\mathbb R^N)}\lesssim \e^{\frac N m}$$ for any $m\geq 2$.
Moreover $\varphi_1$ solves
$$-\Delta\varphi_1+\left(V(x)-3\mu_1\Upsilon^2(x)\right)\varphi_1=a_{12}\Upsilon(x) \Uae\left(\frac x \e\right)+a_{13}\Upsilon(x) \Ube\left(\frac x \e\right).$$

Moreover, since $\nabla\Upsilon(0)=0$ we also have 
\beq\label{stima2}
|\Upsilon(y)-\Upsilon(0)|\lesssim |y|^{2}\|\Upsilon\|_{C^{2}(\mathbb R^N)}\lesssim |y|^{2}.
\eeq
\end{remark}

 \begin{lemma} \label{esistphii}
 Let $i=2, 3$. There exists a unique $\varphi_i\in H^2(\mathbb R^N)$, solution of the equation
 \beq\label{eq2}
 -\Delta\varphi_i+\left(\omega_i-3\mu_i((\Umm)^2+(\Ump)^2)\right)\varphi_i=a_{i1} \Ui\varphi_1(0).\eeq
 Moreover , $\|\varphi_i\|_{H^2(\mathbb R^N)}\lesssim \e^{\frac N 2}$.  \end{lemma}
 \begin{proof}
 The proof can be made as in Lemma 2.3 of \cite{PPVV}.\end{proof}
 \begin{remark}\label{decay}
 We need to understand the asymptotic behavior of $\varphi_i$. Hence we let $\psi$ the radial solution of $$-\Delta\psi+\left(\lambda-3\nu U^2\right)\psi= U$$ where $U$ solves $-\Delta U +\lambda U=\nu U^3$.\\
 Then, reasoning as in Lemma 2.3 of \cite{PPVV}, we have that there exists $R_0>0$ and $0<\gamma<\lambda^2$ such that 
 $$0<\psi(r)<e^{-\sqrt\gamma r}.$$
 Now, let $G(x, y)$ the Green function  associated to the operator $-\Delta+\lambda$ in $\mathbb R^N$. It is well known that 
 $$G(r):=(2\pi)^{-\frac N 2}\left(\frac{\sqrt\lambda}{r}\right)^{\frac{N-2}{2}}\mathfrak K_{\frac N 2-1}\left(\sqrt\lambda r\right)$$ where $\mathfrak K_{\frac N 2-1}$ are the modified Bessel functions. In particular, we have that
 $$G(x, y)=\mathfrak c_0 \frac{e^{-\sqrt{\lambda}|x-y|}}{|x-y|},\,\mbox{if}\, N=3\,\mbox{and}\,\, G(x, y)\lesssim \frac{e^{-\sqrt{\lambda}|x-y|}}{|x-y|^\frac 12},\,\mbox{if}\, N=2.$$
 Let $N=3$ (one can reason similarly for $N=2$). Therefore, we have that 
 
 $$\begin{aligned}
 \psi(x)&=\int_{\mathbb R^3}G(x, y)\left[3\nu U^2(y)\psi(y)+U(y)\right]\, dy\\
 &=\underbrace{\int_{|x-y|<\frac{|x|}{2}}G(x, y)\left[3\nu U^2(y)\psi(y)+U(y)\right]\, dy}_{(I)}+\underbrace{\int_{|y|<\frac{|x|}{2}}G(x, y)\left[3\nu U^2(y)\psi(y)+U(y)\right]\, dy}_{(II)}\\
 &+\underbrace{\int_{\mathbb R^3\setminus\left(B_{\frac{|x|}{2}}(0)\cup B_{\frac{|x|}{2}}(x)\right)}G(x, y)\left[3\nu U^2(y)\psi(y)+U(y)\right]\, dy}_{(III)}.
 \end{aligned}$$
 We analyze the three terms separately and we use the fact the exponential decay of $U$ (see Remark \ref{rema1}).
 \begin{enumerate}
 \item[(I)] Let $|x-y|<\frac{|x|}{2}$. Then $$|y|=|y-x+x|\geq |x|-|x-y|\geq \frac{|x|}{2}$$ and $$|x-y|+|y|\geq |x|-|y|+|y|=|x|.$$ Then
$$\begin{aligned}
(I)&\lesssim\int_{|x-y|<\frac{|x|}{2}}\frac{e^{-\sqrt\lambda|x-y|}}{|x-y|}\left[3\nu \frac{e^{-2\sqrt\lambda|y|}}{|y|^2}e^{-\sqrt\gamma |y|}+\frac{e^{-\sqrt\lambda|y|}}{|y|}\right]\\
&\lesssim \int_{|x-y|<\frac{|x|}{2}}\left[\frac{e^{-\sqrt\lambda(|x-y|+|y|)}}{|x-y|}\frac{e^{-(\sqrt\lambda+\sqrt\gamma)|y|}}{|y|^2}+\frac{e^{-\sqrt\lambda(|x-y|+|y|)}}{|x-y||y|}\right]\, dy\\
&\lesssim \frac{e^{-\sqrt\lambda|x|}}{|x|^2}\int_{|x-y|<\frac{|x|}{2}}\frac{1}{|x-y|}\, dy+\frac{e^{-\sqrt\lambda|x|}}{|x|}\int_{|x-y|<\frac{|x|}{2}}\frac{1}{|x-y|}\, dy \\
&\lesssim |x|e^{-\sqrt\lambda|x|}. \end{aligned}$$

\item[(II)] Let $|y|<\frac{|x|}{2}$. Then $$|x-y|\geq |x|-|y|\geq \frac{|x|}{2}$$ and again $$|x-y|+|y|\geq |x|-|y|+|y|=|x|.$$ Then
$$\begin{aligned}
(II)&\lesssim\int_{|y|<\frac{|x|}{2}}\frac{e^{-\sqrt\lambda|x-y|}}{|x-y|}\left[3\nu \frac{e^{-2\sqrt\lambda|y|}}{|y|^2}e^{-\sqrt\gamma |y|}+\frac{e^{-\sqrt\lambda|y|}}{|y|}\right]\\
&\lesssim \int_{|y|<\frac{|x|}{2}}\left[\frac{e^{-\sqrt\lambda(|x-y|+|y|)}}{|x-y|}\frac{e^{-(\sqrt\lambda+\sqrt\gamma)|y|}}{|y|^2}+\frac{e^{-\sqrt\lambda(|x-y|+|y|)}}{|x-y||y|}\right]\, dy\\
&\lesssim \frac{e^{-\sqrt\lambda|x|}}{|x|}\int_{|y|<\frac{|x|}{2}}\frac{1}{|y|^2}\, dy+\frac{e^{-\sqrt\lambda|x|}}{|x|}\int_{|y|<\frac{|x|}{2}}\frac{1}{|y|}\, dy \\
&\lesssim |x|e^{-\sqrt\lambda|x|}. \end{aligned}$$

\item[(III)] Let $|y|>\frac{|x|}{2}$ and $|x-y|>\frac{|x|}{2}$.\\
Let
 $A_0:=\left(B_{\frac{|x|}{2}}(0)\cup B_{\frac{|x|}{2}}(x)\right)^c$. Then $$A_0:=\underbrace{\left(B^c_{|x|}(0)\cap B_{\frac{|x|}{2}}^c(x)\right)}_{:=A_1}\cup \underbrace{\left[\left(B_{|x|}(0)\setminus B_{\frac{|x|}{2}}(0)\right)\setminus \left( B_{|x|}(0)\cap  B_{\frac{|x|}{2}}(x)\right)\right]}_{:=A_2}$$ with $A_1\cap A_2=\emptyset.$ We remark that in $A_2$ $|x-y|>\frac{|x|}{2}$. Then $$\begin{aligned}
(III)&\lesssim\int_{A_1}\frac{e^{-\sqrt\lambda|x-y|}}{|x-y|}\left[3\nu \frac{e^{-2\sqrt\lambda|y|}}{|y|^2}e^{-\sqrt\gamma |y|}+\frac{e^{-\sqrt\lambda|y|}}{|y|}\right]\\
&+\int_{A_2}\frac{e^{-\sqrt\lambda|x-y|}}{|x-y|}\left[3\nu \frac{e^{-2\sqrt\lambda|y|}}{|y|^2}e^{-\sqrt\gamma |y|}+\frac{e^{-\sqrt\lambda|y|}}{|y|}\right]\\
&\lesssim \frac{e^{-\sqrt\lambda|x|}}{|x|^2}\int_{A_1}\frac{e^{-\sqrt\lambda|x-y|}}{|x-y|}e^{-(\sqrt\lambda+\sqrt\gamma)|y|}\, dy +\frac{e^{-\sqrt\lambda|x|}}{|x|}\int_{A_1}\frac{e^{-\sqrt\lambda|x-y|}}{|x-y|}\, dy\\
&+\frac{e^{-\sqrt\lambda|x|}}{|x|}\int_{A_2}\frac{e^{-(\sqrt\lambda+\sqrt\gamma)|y|}}{|y|^2}\, dy +\frac{e^{-\sqrt\lambda|x|}}{|x|}\int_{A_2}\frac{1}{|y|}\, dy\\
&\lesssim \frac{e^{-\sqrt\lambda|x|}}{|x|}\int_{\mathbb R^3}\frac{e^{-\sqrt\lambda|y|}}{|y|}\, dy+\frac{e^{-\sqrt\lambda|x|}}{|x|}\int_{B_{|x|}(0)\setminus B_{\frac{|x|}{2}}(0)}\frac{1}{|y|}\, dy\\ &\lesssim \frac{1}{|x|}e^{-\sqrt\lambda|x|}. \end{aligned}$$
 \end{enumerate}
 At the end we get that $$\psi(x)\lesssim |x| e^{-\sqrt\lambda |x|}.$$
 Hence, since $$\varphi_i (x)=a_{i1}\varphi_1(0)\left[\psi\left(x+\frac{P_i}{\e}\right)+\psi\left(x-\frac{P_i}{\e}\right)\right]$$ then we conclude that there is $R_0>0$ such that  $$|\varphi_i(x)|\lesssim \e^{\frac N 2}\left(\left(\left|x+\frac{P_i}{\e}\right|^{\frac{N-1}{2}}\right)e^{-\sqrt{\omega_i}\left|x+\frac{P_i}{\e}\right|}+\left(\left|x-\frac{P_i}{\e}\right|\right)^{\frac{N-1}{2}}e^{-\sqrt{\omega_i}\left|x-\frac{P_i}{\e}\right|}\right),\quad \mbox{for}\,\, |x|\geq R_0.$$
 \end{remark}

 \subsection{Rewriting the problem}
 Plugging the ansatz into  \eqref{DS1} we obtain the following equivalent system

 \beq\label{DS2}
 \mathcal L(\Phi)=\mathcal E+\mathcal N(\Phi)\eeq 
where the linear operator $\mathcal L(\Phi)=(\mathcal L_1, \mathcal L_2, \mathcal L_3)$ is defined by
 \begin{equation}\label{lin1}  \begin{aligned}
 \mathcal L_1&:=-\Delta\Phi_1+\left(V(x)-3\mu_1(\Upsilon+\varphi_1)^2\right)\Phi_1-\sum_{i=2}^3a_{1i}(\Upsilon +\varphi_1)\Phi_i\left(\frac x \e\right)\\
 &-\sum_{i=2}^3a_{1i}\Phi_1\left(\Ui+\varphi_i\right)\left(\frac x \e\right)
 \end{aligned}\end{equation}
 while
  \begin{equation}\label{lin2}
 \begin{aligned}
 \mathcal L_2&:=-\Delta\Phi_2+\left(W_2(\e x)-3\mu_2(\Uae+\varphi_2)^2\right)\Phi_2-a_{21}(\Uae+\varphi_2)\Phi_1(\e x)\\ &-\beta a_{23}(\Ube+\varphi_3)\Phi_2-a_{21}(\Upsilon+\varphi_1)(\e x)\Phi_2
 \end{aligned}\end{equation}
 and 
 \begin{equation}\label{lin3}
 \begin{aligned}
 \mathcal L_3&:=-\Delta\Phi_3+\left(W_3(\e x)-3\mu_3(\Ube+\varphi_3)^2\right)\Phi_3-a_{31}(\Ube+\varphi_3)\Phi_1(\e x)\\ &-\beta a_{32}(\Ube+\varphi_3)\Phi_2-a_{31}(\Upsilon+\varphi_1)(\e x)\Phi_3
 \end{aligned}\end{equation}
 the error term $\mathcal E=(\mathcal E_1, \mathcal E_2, \mathcal E_3)$ is defined by
 \beq\label{E1}\mathcal E_1:=3\mu_1 \Upsilon\varphi_1^2+\mu_1\varphi_1^3+\sum_{i=2}^3 a_{1i}(\Upsilon+\varphi_1)\varphi_i\left(\frac x \e\right)+\sum_{i=2}^3a_{1i}\varphi_1\Ui\left(\frac x \e\right).\eeq
and
\beq\label{E2}\begin{aligned}
\mathcal E_2&:=3\mu_2 \left((\Uap)^2\Uam+(\Uam)^2\Uap\right)+6\mu_2 \Uap\Uam\varphi_2+3\mu_2\Uae\varphi_2^2+\mu_2\varphi_2^3\\
&+(\omega_2-\omega_2(\e x))(\Uae+\varphi_2)+a_{21}\Uae(\varphi_1(\e x)-\varphi_1(0))+a_{21}\varphi_1(\e x)\varphi_2\\
&+\beta a_{23}(\Uae+\varphi_2)(\Ube+\varphi_3)
.\end{aligned}\eeq
\beq\label{E3}\begin{aligned}
\mathcal E_3&:=3\mu_3 \left((\Ubp)^2\Ubm+(\Ubm)^2\Ubp\right)+6\mu_2 \Ubp\Ubm\varphi_3+3\mu_3\Ube\varphi_3^2+\mu_3\varphi_3^3\\
&+(\omega_3-\omega_3(\e x))(\Ube+\varphi_3)+a_{31}\Ube(\varphi_1(\e x)-\varphi_1(0))+a_{31}\varphi_1(\e x)\varphi_3\\
&+\beta a_{32}(\Uae+\varphi_2)(\Ube+\varphi_3)
.\end{aligned}\eeq
 The nonlinear term $\mathcal N=(\mathcal N_1,  \mathcal N_2, \mathcal N_3)$ is defined by
 $$\mathcal N_1(\Phi_1, \Phi_2, \Phi_3):=3\mu_1(\Upsilon+\varphi_1)\Phi_1^2+\mu_1\Phi_1^3+\sum_{i=2}^3a_{1i}\Phi_1\Phi_i\left(\frac x \e\right)$$
 $$\mathcal N_2(\Phi_1,\Phi_2,\Phi_3):=3\mu_2(\Uae+\varphi_2)\Phi_2^2+\mu_2\Phi_2^3+a_{21}\Phi_1(\e x)\Phi_2+\beta a_{23}\Phi_2\Phi_3-\beta a_{23}(\Uae+\varphi_2)\Phi_3$$ and
 $$\mathcal N_3(\Phi_1,\Phi_2,\Phi_3):=3\mu_3(\Ube+\varphi_3)\Phi_3^2+\mu_3\Phi_3^3+a_{31}\Phi_1(\e x)\Phi_3+\beta a_{32}\Phi_2\Phi_3-\beta a_{32}(\Uae+\varphi_2)\Phi_3.$$
Let $\mathcal K:={\rm span}\left\{(0, \mathcal Z_\e, \mathcal Y_\e)\right\}$ and let $\Pi$ the projection of $L^2(\mathbb R^N)$ into $\mathcal K$ and $\Pi^\bot$ the projection of $L^2(\mathbb R^N)$  into $\mathcal K^\bot$. Then the system \eqref{DS2} is equivalent to 
\begin{equation}\label{DSnonlin1}
\left\{\begin{aligned}
&\Pi\left\{ \mathcal L(\Phi)-\mathcal E-\mathcal N(\Phi)\right\}=0\\
&\Pi^\bot\left\{ \mathcal L(\Phi)-\mathcal E-\mathcal N(\Phi)\right\}=0.\end{aligned}\right.\end{equation}
The strategy is the following:
  \begin{itemize}
  \item First we look for $(\Phi_1, \Phi_2, \Phi_3)\in \mathcal K^\bot$ that solves the second equation in \eqref{DSnonlin1}. 
\item Then, if we denote by $(\Phi_1, \Phi_2, \Phi_3)$ the solution of the second equation already founded, then the first equation becomes
\begin{equation}\label{DSnonlin}\left\{\begin{aligned}&\mathcal L_2(\Phi_1, \Phi_2, \Phi_3)-\mathcal E_2-\mathcal N_2(\Phi_1, \Phi_2, \Phi_3)=\mathfrak c_2 \mathcal Z_\e\\
&\mathcal L_3(\Phi_1, \Phi_2, \Phi_3)-\mathcal E_3-\mathcal N_3(\Phi_1, \Phi_2, \Phi_3)=\mathfrak c_3 \mathcal Y_\e\end{aligned}\right.\end{equation} where

$$\mathfrak c_2:=\frac{\left(\mathcal L_2(\Phi_1, \Phi_2, \Phi_3)-\mathcal E_2-\mathcal N_2(\Phi_1, \Phi_2, \Phi_3), \mathcal Z_\e\right)_{L^2(\mathbb R^N)}}{\|\mathcal Z_\e\|_{L^2(\mathbb R^N)}}$$ and 
$$\mathfrak c_3:=\frac{\left(\mathcal L_3(\Phi_1, \Phi_2, \Phi_3)-\mathcal E_3-\mathcal N_3(\Phi_1, \Phi_2, \Phi_3), \mathcal Y_\e\right)_{L^2(\mathbb R^N)}}{\|\mathcal Y_\e\|_{L^2(\mathbb R^N)}}$$
so, in order to find a solution we need to prove that $\mathfrak c_2$ and $\mathfrak c_3$ are zero.
  \end{itemize}
  \section{The solution of the second equation of \eqref{DSnonlin1}}\label{ausiliaria}
  In this section we find a solution $(\Phi_1, \Phi_2, \Phi_3)\in\mathcal K^\bot$ of the second equation of \eqref{DSnonlin1}.
 \subsection{The linear problem.}
First we study the invertibility of the linear operator $\mathcal L$. Indeed, we have the following result.

\begin{lemma}\label{inv}
    There exist $c>0$ and $\e_0>0$ such that for every $\e\in (0,\e_0)$, for every\\ $a_{i1}<0, \, i=2,3$ $$\|\Pi^\bot \mathcal{L} (\Phi_1,\Phi_2, \Phi_3)\|_{L^2(\RN)\times L^2(\RN)\times L^2(\RN)} \geq c \|(\Phi_1,\Phi_2, \Phi_3)\|_X \quad \forall (\Phi_1,\Phi_2, \Phi_3)\in \mathcal{K}^\bot.$$
\end{lemma}
\begin{proof}
    We argue by contradiction and suppose that there exist $\e_n\longrightarrow0 $ and \\ $(\Phi_{1,n},\Phi_{2,n}, \Phi_{3,n})\in \mathcal{K}^\bot$ with $\|(\Phi_{1,n},\Phi_{2,n}, \Phi_{3,n})\|_{X}=1$ such that
     \beq\label{P}
 \left\{\begin{aligned}
 &\mathcal L_1(\Phi_1, \Phi_2, \Phi_3)=f_n\\
  &\mathcal L_2(\Phi_1, \Phi_2, \Phi_3)=g_n+\mathfrak c_{2, n} \mathcal Z_{\e_n}\\
    &\mathcal L_3(\Phi_1, \Phi_2, \Phi_3)=h_n+\mathfrak c_{3, n} \mathcal Y_{\e_n}
  \end{aligned}\right.\eeq

where
$$f_n, g_n, h_n \longrightarrow 0 \quad \mbox{in}\,  L^2(\RN)$$
$$\int_{\RN} g_n \mathcal Z_{\e_n}\, dx=\int_{\RN} h_n \mathcal Y_{\e_n}\, dx=0$$
and $\mathfrak c_{2, n}, \mathfrak c_{3, n} \in \mathbb R.$\\
\textbf{STEP 1:}\, {$\Phi_{1,n}\longrightarrow0$ strongly in $H^2_V(\RN).$}\\
To prove this step, we can reason as in Lemma 2.5 of \cite{PPVV}.\\ We only remark that, by Sobolev embeddings, it follows that $\Phi_{1, n}\longrightarrow 0$ in $L^\infty(\RN).$\\

\textbf{STEP 2:} $\mathfrak c_{2, n}, \mathfrak c_{3, n} \longrightarrow 0.$\\
We show only that $\mathfrak c_{2, n}\to 0$ since similarly, it follows that $\mathfrak c_{3, n}\longrightarrow 0.$\\\\ To this aim, we test the second equation with $ \mathcal Z_{\e_n}$ recalling that $ \mathcal Z_{\e_n}$ solves 
$$-\Delta  \mathcal Z_{\e_n}+\omega_2  \mathcal Z_{\e_n}=3\mu_2 \left((\mathcal{U}_{2,P_{2_n}})^2\partial_1\mathcal{U}_{2,P_{2_n}}-(\mathcal{U}_{2,-P_{2_n}})^2\partial_1\mathcal{U}_{2,-P_{2_n}}\right).$$
Therefore, we get, since $\int_{\RN} g_n \mathcal Z_{\e_n}=0$, 
\begin{equation*}
    \begin{aligned}
        \mathfrak c_{2, n} \int_{\RN} \mathcal Z_{\e_n}^2&=\int_{\RN}\left[\omega_2(\e_n x) -\omega_2\right] \mathcal Z_{\e_n}\Phi_{2,n}\\
        &-3\mu_2\int_{\RN} \left[\mathcal{U}^2_{2,\e_n}  \mathcal Z_{\e_n}- ((\mathcal{U}_{2,P_{2_n}})^2\partial_1\mathcal{U}_{2,P_{2_n}}-(\mathcal{U}_{2,-P_{2_n}})^2\partial_1\mathcal{U}_{2,-P_{2_n}}) \right] \Phi_{2,n}\\
        &-a_{21}\int_{\RN}(\mathcal{U}_{2,\e_n}+\varphi_{2,n})\Phi_{1,n}(\e_n x) \mathcal Z_{\e_n}-a_{21}\int_{\RN}\varphi_{1,n}(\e_n x)\Phi_{2,n}  \mathcal Z_{\e_n}\\
        &-\beta a_{23}\int_{\RN}(\mathcal{U}_{3,\e_n}+\varphi_{3,n})\Phi_{2,n}  \mathcal Z_{\e_n}-3\mu_2 \int_{\RN}\varphi_{2,n}(\varphi_{2,n}+2\mathcal{U}_{2,\e_n})\Phi_{2,n}  \mathcal Z_{\e_n}\\
        &\lesssim \|[\omega_2 (\e_n x)-\omega_2] \mathcal Z_{\e_n}\|_{L^2} \|\Phi_{2,n}\|_{L^2}\\\\
        &+\|[\mathcal{U}^2_{2,\e_n}  \mathcal Z_{\e_n}- ((\mathcal{U}_{2,P_{2_n}})^2\partial_1\mathcal{U}_{2,P_{2_n}}-(\mathcal{U}_{2,-P_{2_n}})^2\partial_1\mathcal{U}_{2,-P_{2_n}})]\|_{L^2}\|\Phi_{2,n}\|_{L^2}\\\\
        &+\|\Phi_{1,n}\|_{L^\infty}+\|\varphi_{1,n}\|_{L^\infty}+ \|\varphi_{2,n}\|_{L^\infty}+ \|\varphi_{3,n}\|_{L^\infty}+\left(\int_{\RN}|\mathcal{U}_{3,\e_n} \mathcal Z_{\e_n} |^2\right)^{\frac 12}   \end{aligned}
\end{equation*}
using the exponential decay of $\mathcal{U}$ and its derivatives.\\
We get that
$$|\omega_2(y)-\omega_2|\lesssim  |y|\quad {if}\quad |y|\leq \sigma$$
for some $\sigma>0$ and so 
$$\|[\omega_2(\e_n x)-\omega_2]\mathcal Z_{\e_n}\|_{L^2}=o(1).$$
Moreover, a direct computation and Lemma \ref{ACR} shows that
$$\|[\mathcal{U}^2_{2,\e_n} Z_{\e_n}- ((\mathcal{U}_{2,P_{2_n}})^2\partial_1\mathcal{U}_{2,P_{2_n}}-(\mathcal{U}_{2,-P_{2_n}})^2\partial_1\mathcal{U}_{2,-P_{2_n}})]\|_{L^2}=o(1).$$
It's possible to show that all the other integral terms on the right hand side tend to zero by applying Lemma \ref{hatvarphi}, Lemma \ref{esistphii}, Step 1 and by Sobolev embeddings.\\
Moreover, by using Lemma \ref{ACR} we get that
\begin{equation}\label{utile}\begin{aligned}&\int_{\RN} |\mathcal{U}_{3,\e_n} \mathcal Z_{\e_n} |^2 \lesssim \int_{\RN} |\mathcal{U}_{3,\e_n} \mathcal U_{2, \e_n} |^2\\ &\lesssim \left\{\begin{aligned} &e^{-2\min\{\sqrt{\omega_2}, \sqrt{\omega_3}\}\frac{|P_2\pm P_3|}{\e_n}}\left(\frac{|P_2\pm P_3|}{\e_n}\right)^{-N+1}\quad &\mbox{if}\, \omega_2\neq \omega_3\\ &e^{-2\min\{\sqrt{\omega_2}, \sqrt{\omega_3}\}\frac{|P_2\pm P_3|}{\e_n}}\left(\frac{|P_2\pm P_3|}{\e_n}\right)^{-\frac 12}\quad &\mbox{if}\, \omega_2=\omega_3\,\mbox{and}\, N=2\\
&e^{-2\min\{\sqrt{\omega_2}, \sqrt{\omega_3}\}\frac{|P_2\pm P_3|}{\e_n}}\left(\frac{|P_2\pm P_3}{\e_n}\right)^{-N+1}\log\left(\frac{|P_2\pm P_3|}{\e_n}\right)\quad &\mbox{if}\, \omega_2=\omega_3\,\mbox{and}\, N=3\end{aligned}\right.\\
&=o(1)\end{aligned}
\end{equation} as $\e_n\to 0$ since $\frac{|P_2\pm P_3|}{\e_n}\to+\infty$ as $\e_n\to 0$.\\
Finally, since it is immediate to check that $\|\mathcal Z_{\e_n}\|_{L^2}=C+o(1)$ for some $C>0,$ we deduce that $\mathfrak c_{2, n}=o(1).$\\ Similarly $\mathfrak c_{3, n}=o(1)$.   \\\\ \textbf{STEP 3:}\, Let us now introduce the sequences:
$$\Tilde{\Phi}_{-i,n}(x):=\Phi_{i,n}\left(x-\frac{P_{i,\e_n}}{\e_n}\right), \quad \Tilde{\Phi}_{+i,n}(x):=\Phi_{i,n}\left(x+\frac{P_{i,\e_n}}{\e_n}\right)\quad i=2, 3.$$
Then (up to subsequences) $\Tilde{\Phi}_{\pm i ,n}\rightharpoonup 0$ weakly in $H^1(\RN)$ and strongly in $L^2_{loc}(\RN).$\\\\
Again for proving this step we can reason as in Lemma 2.5 of \cite{PPVV}.\\\\
\textbf{STEP 4}\, Let us prove that a contradiction arise.\\
First, let us prove that $\|\Phi_{2,n}\|_{L^2}=o(1)$ (analogously $\|\Phi_{3,n}\|_{L^2}=o(1)$).\\
By testing the second equation with $\Phi_{2,n}$ and recalling that $a_{i1}<0$ we deduce that
\begin{equation*}
    \begin{aligned}
        \int_{\RN} |\nabla \Phi_{2,n}|^2+W_2(\e_n x) \Phi_{2,n}^2&= \int_{\RN} 3\mu_2 \mathcal{U}^2_{2,\e_n}\Phi_{2,n}^2 +a_{21}\int_{\RN}\Upsilon(\e_n x)\Phi_{2,n}^2\\
        &+a_{21}\int_{\RN}\varphi_{1,n}(\e_n x)\Phi_{2,n}^2+ 3\mu_2 \int_{\RN}\varphi_{2,n}(\varphi_{2,n}+2\mathcal{U}_{2,\e_n})\Phi_{2,n}^2\\
        &+a_{21}\int_{\RN}(\mathcal{U}_{2,\e_n}+\varphi_{2,n})\Phi_{1,n}(\e_n x) \Phi_{2,n}\\
        &+\beta a_{23}\int_{\RN}(\mathcal{U}_{3,\e_n}+\varphi_{3,n})\Phi_{3,n} \Phi_{2,n}+\int_{\RN} g_n\Phi_{2,n}\lesssim\\
        &\lesssim \int_{\RN} 3\mu_2 \mathcal{U}^2_{2,\e_n}\Phi_{2,n}^2+o(1)
    \end{aligned}
\end{equation*}
where we have repeatedly applied Lemma \ref{hatvarphi}, Lemma \ref{esistphii}, \eqref{utile} and that $g_n\longrightarrow 0$ strongly in $L^2(\RN).$\\
Concerning the last term, we have that 
\begin{equation*}
    \begin{aligned}
        \int_{\RN}  \mathcal{U}^2_{2,\e_n}\Phi_{2,n}^2&= \int_{\RN} \mathcal{U}^2_{2,P_{2_n}}\Phi_{2,n}^2+\int_{\RN} \mathcal{U}^2_{2,-P_{2_n}}\Phi_{2,n}^2 +2\int_{\RN} \mathcal{U}_{2,P_{2_n}}\mathcal{U}_{2,-P_{2_n}}\Phi_{2,n}^2\lesssim\\
        &\lesssim\int_{\RN} \mathcal{U}^2_2 \Tilde{\Phi}^2_{+2,n}+\int_{\RN}\mathcal{U}^2_2 \Tilde{\Phi}^2_{-2,n}+\|\Phi_{2,n}\|_{L^\infty}\int_{\RN} \mathcal{U}_{2,P_{2_n}}\mathcal{U}_{2,-P_{2_n}}=o(1)
    \end{aligned}
\end{equation*}
because $\Tilde{\Phi}_{\pm 2, n}\longrightarrow0$ strongly in $L^2_{loc}(\RN)$ (as shown in the previous step) and $\mathcal{U}_2$ decays exponentially. \\
This implies that 
$\Phi_{2,n}\longrightarrow0$ strongly in $H^1(\RN)$ thank to $\bf {(W_2)}$ (like $\Phi_{3,n}$).\\
Finally, let us prove that a contradiction arises by showing that also $\Phi_{2,n}\longrightarrow0$ strongly in $H^2_{W_2,\e_n}(\RN)$ (and too $\Phi_{3,n}\longrightarrow0$ strongly in $H^2_{W_3,\e_n}(\RN)).$\\
In order to show this, it is enough to use the hypothesis $\bf {(V_2)}$ and to check that the $L^2(\RN)-$ norm of the right hand side of the second equation in (\ref{P}) goes to zero. \\
Indeed, by Lemma \ref{hatvarphi}, taking into account that 
$$\|\Phi_{2,n}\|_{L^2}\longrightarrow0,\|\Phi_{3,n}\|_{L^2}\longrightarrow0, \|\Phi_{1,n}\|_{L^\infty}\longrightarrow0 $$
it results 
\begin{equation*}
    \begin{aligned}
        \|R.H.S\|_{L^2}&\lesssim \|3\mu_2 \mathcal{U}^2_{2,\e_n}\Phi_{2,n}\|_{L^2}+\|3\mu_2 \varphi_{2,n}(\varphi_{2,n}+2\mathcal{U}_{2,\e_n})\Phi_{2,n}\|_{L^2}\\
        &+\|(\Upsilon+\varphi_{1,n})(\e_n x)\Phi_{2,n}\|_{L^2}+\|(\mathcal{U}_{3,\e_n}+\varphi_{3,n})\Phi_{2,n}\|_{L^2}+\|(\mathcal{U}_{2,\e_n}+\varphi_{2,n})\Phi_{1,n}(\e_n x)\|_{L^2}\\
        &+ \|g_n\|_{L^2}+|\mathfrak c_{2, n}| \|\mathcal Z_{\e_n}\|_{L^2}\lesssim\\
        &\lesssim \|\Phi_{2,n}\|_{L^2}+\|\Phi_{1,n}\|_{L^\infty}+\|\Phi_{3,n}\|_{L^2}+ \|g_n\|_{L^2}+|\mathfrak c_{2, n}|=o(1).
    \end{aligned}
\end{equation*}
\end{proof}

 \subsection{The error term} In this section we evaluate the size of the error term $(\E_1, \E_2, \E_3)$. \\ First we introduce some notations.\\ We let $\mathfrak m_0:=\min\{\omega_2, \omega_3\}$ and we let $$|\rho_\e|:=\sqrt{\rho_{2, \e}^2+\rho_{3, \e}^2}=|P_3-P_2|=|P_3+P_2|.$$
 Moreover, we define
 $$\begin{aligned} \eta_\e &=  \left\{\begin{aligned}&e^{-\sqrt{\mathfrak m_0}\frac{|\rho_\e|}{\e}}\left(\frac{|\rho_\e|}{\e}\right)^{-\frac {N-1}{2}}\quad &\mbox{if}\,\, \omega_2\neq\omega_3\,\, \mbox{and}\,\, N=2, 3\\
&e^{-\sqrt{\mathfrak m_0}\frac{|\rho_\e|}{\e}}\left(\frac{|\rho_\e|}{\e}\right)^{-\frac 14}\quad &\mbox{if}\,\, \omega_2=\omega_3\,\, \mbox{and}\,\, N=2\\
&e^{-\sqrt{\mathfrak m_0}\frac{|\rho_\e|}{\e}}\left(\frac{|\rho_\e|}{\e}\right)^{-1}\log^{\frac 12}\frac{|\rho_\e|}{\e}\quad &\mbox{if}\,\, \omega_2=\omega_3\,\, \mbox{and}\,\, N=3\end{aligned}\right.\end{aligned}$$
We prove the following lemma.
  \begin{lemma}\label{errore}
 There exist $\e_0>0$ such that for every $\e\in (0, \e_0)$  $$\|(\E_1, \E_2, \E_3)\|_{L^2(\mathbb R^N)\times L^2(\mathbb R^N)\times L^2(\mathbb R^N)}\lesssim \e^{2}|\ln\e|^2+\sum_{i=2}^3e^{-2\sqrt{\omega_i}\frac{\rho_{i, \e}}{\e}}\left(\frac{\rho_{i, \e}}{\e}\right)^{-\frac{N-1}{2}}+|\beta|\eta_\e$$ 
 \end{lemma}
 \begin{proof}
 Let us start by evaluating $\mathcal E_1$ given in \eqref{E1}. We have
 $$\begin{aligned}\|\mathcal E_1\|_{L^2(\mathbb R^N)}&\lesssim \left\|\Upsilon\varphi_1^2\right\|_{L^2(\mathbb R^N)}+\|\varphi_1^3\|_{L^2(\mathbb R^N)}+\sum_{i=2}^3\left\|(\Upsilon+\varphi_1)\varphi_i\left(\frac x \e\right)\right\|_{L^2(\mathbb R^N)}\\
 &+\sum_{i=2}^3\left\|\varphi_1 \Ui\left(\frac x \e\right)\right\|_{L^2(\mathbb R^N)}.\end{aligned}$$
 Now, by using Lemma \ref{hatvarphi} and Sobolev embedding we get that $\|\varphi_1\|_\infty\lesssim \e^{\frac N 2}.$ Moreover by Lemma \ref{esistphii} we get also that $\|\varphi_i\|_{L^2(\mathbb R^N)}\lesssim \e^{\frac N 2}.$ Then
 $$\left\|\Upsilon\varphi_1^2\right\|_{L^2(\mathbb R^N)}\lesssim \e^{N },\quad  \|\varphi_1^3\|_{L^2(\mathbb R^N)}\lesssim \e^{3\frac N 2}$$
 Moreover
 $$\left\|(\Upsilon+\varphi_1)\varphi_i\left(\frac x \e\right)\right\|_{L^2(\mathbb R^N)}\lesssim \e^{\frac N 2}\|\varphi_i\|_{L^2(\mathbb R^N)}+\e^N \|\varphi_i\|_{L^2(\mathbb R^N)}\lesssim \e^{N} $$
 and $$\left\|\varphi_1 \Ui\left(\frac x \e\right)\right\|_{L^2(\mathbb R^N)}\lesssim \e^{\frac N 2}\left\|\Ui\left(\frac x \e\right)\right\|_{L^2(\mathbb R^N)}\lesssim \e^N .$$
 Hence
 $$\|\E_1\|_{L^2(\mathbb R^N)}\lesssim \e^N.$$
For what concerning $\E_2$ defined in \eqref{E2} we get that

 $$\begin{aligned}\|\mathcal E_2\|_{L^2(\mathbb R^N)}&\lesssim\|(\Uap)^2\Uam\|_{L^2(\mathbb R^N)}+\|(\Uam)^2\Uap\|_{L^2(\mathbb R^N)}+\|\Uam\Uap\varphi_2\|_{L^2(\mathbb R^N)}+\| \Uae\varphi_2^2\|_{L^2(\mathbb R^N)}\\
 &+\|\varphi_2^3\|_{L^2(\mathbb R^N)}+\|(\omega_2-\omega_2(\e x))(\Uae+\varphi_2)\|_{L^2(\mathbb R^N)}+\|\Uae(\varphi_1(\e x)-\varphi_1(0))\|_{L^2(\mathbb R^N)}\\ &+\|\varphi_2\varphi_1(\e x)\|_{L^2(\mathbb R^N)}+|\beta|\|(\Uae+\varphi_2)(\Ube+\varphi_3)\|_{L^2(\mathbb R^N)}.\end{aligned}$$
 Reasoning as in Proposition 3.3 of \cite{PPVV} then
 $$\|(\omega_2-\omega_2(\e x))(\Uae+\varphi_2)\|_{L^2(\mathbb R^N)}\lesssim \rho_{2,\e}^2$$
 while
$$ \|(\Uap)^2\Uam\|_{L^2(\mathbb R^N)}+\|(\Uam)^2\Uap\|_{L^2(\mathbb R^N)}\lesssim e^{-2\sqrt{\omega_2}\frac{\rho_{2, \e}}{\e}}\left(\frac{\rho_{2, \e}}{\e}\right)^{-\frac{N-1}{2}}$$ and
$$\|\Uam\Uap\varphi_2\|_{L^2(\mathbb R^N)}\lesssim \|\varphi_2\|_\infty \|\Uam\Uap\|_{L^2(\mathbb R^N)}=o\left(e^{-2\sqrt{\omega_2}\frac{\rho_{2, \e}}{\e}}\left(\frac{\rho_{2, \e}}{\e}\right)^{-\frac{N-1}{2}}\right).$$
Now
$$\|\Uae \varphi_2^2\|_{L^2(\mathbb R^N)}\lesssim \e^N,\qquad \|\varphi_2^3\|_{L^2(\mathbb R^N)}\lesssim \e^{\frac 32 N},\qquad \|\varphi_2\varphi_1(\e x)\|_{L^2(\mathbb R^N}\lesssim \e^N.$$ By using Remark \ref{stimaphi1} then
$$\|\Uae (\varphi_1(\e x)-\varphi_1(0)\|_{L^2(\mathbb R^N)}\lesssim \e^2\left(\int_{\mathbb R^N}\left(\Ua(x)\left|x+\frac{P_2}{\e}\right|^{2-\frac N m}\right)^2\, dx\right)^{\frac 12}\lesssim \e^2\left(\frac{\rho_{2, \e}}{\e}\right)^{2-\frac N m}=o(\rho_{2, \e}^2).$$
It remains to evaluate the new term. Let $\mathfrak m_0:=\min\{\omega_2, \omega_3\}$. Then
$$\begin{aligned}\|\Uae\Ube\|_{L^2(\mathbb R^N)}&\lesssim \|\Uap\Ubp\|_{L^2(\mathbb R^N)}+\|\Uam\Ubm\|_{L^2(\mathbb R^N)}\\
&+\|\Uam\Ubp\|_{L^2(\mathbb R^N)}+\|\Uap\Ubm\|_{L^2(\mathbb R^N)}\end{aligned}$$
We estimate only the first term in the previous inequality since the others can be treated in a similar way. We have that (by using Lemma \ref{ACR})
$$\begin{aligned} \|\Uap\Ubp\|_{L^2(\mathbb R^N)}&\lesssim  \left\{\begin{aligned}&e^{-\sqrt{\mathfrak m_0}\frac{|\rho_\e|}{\e}}\left(\frac{|\rho_\e|}{\e}\right)^{-\frac {N-1}{2}}\quad &\mbox{if}\,\, \omega_2\neq\omega_3\,\, \mbox{and}\,\, N=2, 3\\
&e^{-\sqrt{\mathfrak m_0}\frac{|\rho_\e|}{\e}}\left(\frac{|\rho_\e|}{\e}\right)^{-\frac 14}\quad &\mbox{if}\,\, \omega_2=\omega_3\,\, \mbox{and}\,\, N=2\\
&e^{-\sqrt{\mathfrak m_0}\frac{|\rho_\e|}{\e}}\left(\frac{|\rho_\e|}{\e}\right)^{-1}\log^{\frac 12}\frac{|\rho_\e|}{\e}\quad &\mbox{if}\,\, \omega_2=\omega_3\,\, \mbox{and}\,\, N=3\end{aligned}\right.\\
&:=\eta_\e.\end{aligned}$$
 By using Lemma \ref{esistphii} and Remark \ref{decay} it is easy to show that
 $$\|\Uae\varphi_3\|_{L^2(\mathbb R^N)}=o(\eta_\e),\quad \|\Ube\varphi_2\|_{L^2(\RN)}=o(\eta_\e)$$ and 
$$\|\varphi_2\varphi_3\|_{L^2(\RN)}\lesssim \e^N.$$
 Putting together the various estimates we get that
 $$\|\E_2\|_{L^2(\RN)}\lesssim \e^N+\rho_{2, \e}^2 + e^{-2\sqrt{\omega_2}\frac{\rho_{2, \e}}{\e}}\left(\frac{\rho_{2, \e}}{\e}\right)^{-\frac{N-1}{2}}+|\beta|\eta_\e$$ and 
 $$\|\E_3\|_{L^2(\RN)}\lesssim \e^N+\rho_{3, \e}^2 + e^{-2\sqrt{\omega_3}\frac{\rho_{3, \e}}{\e}}\left(\frac{\rho_{3, \e}}{\e}\right)^{-\frac{N-1}{2}}+|\beta|\eta_\e.$$
 Now, since $\rho_{i, \e}\in\mathcal D_{i, \e}^j$ it is easy to see that the thesis follows.
\end{proof}
We are ready to show the existence of $(\Phi_1, \Phi_2, \Phi_3)\in\mathcal K^\bot$. \\ Let $$\tau_\e:=\e^{2}|\ln\e|^2+\sum_{i=2}^3e^{-2\sqrt{\omega_i}\frac{\rho_{i, \e}}{\e}}\left(\frac{\rho_{i, \e}}{\e}\right)^{-\frac{N-1}{2}}+|\beta|\eta_\e.$$  
\begin{proposition}\label{nonlin}
There exits $\e_0>0$ such that for every $\e\in(0, \e_0)$ there exists $\beta^*_\e>0$ such that for every $|\beta|\leq \beta^*_\e$, there exists a unique solution $(\Phi_1, \Phi_2, \Phi_3)\in\mathcal K^\bot $ of the system \eqref{DSnonlin1}. Furthermore, $$\|(\Phi_1, \Phi_2, \Phi_3)\|_X\lesssim \tau_\e$$ for some $\sigma>0$.
\end{proposition}
\begin{proof}
The proof of this proposition can be made as in \cite{PPVV} (see Proposition 3.4) by applying a standard contraction mapping argument. We only underline the reason why we have to choose $\beta$ sufficiently small.\\ Indeed, we let the ball 
$$B_\e:=\{(\Phi_1, \Phi_2, \Phi_3)\in X\,;\, \|(\Phi_1, \Phi_2, \Phi_3)\|_X \leq R \tau_\e \}$$ where $R$ is a suitable positive constant to be chosen.\\ It is easy to see that
$$\begin{aligned}\|\mathcal N_1(\Phi_1, \Phi_2, \Phi_3)\|_{L^2(\mathbb R^N)}\lesssim \|(\Phi_1, \Phi_2, \Phi_3)\|_X^2\end{aligned}$$  while
$$\begin{aligned}\|\mathcal N_2(\Phi_1, \Phi_2, \Phi_3)\|_{L^2(\mathbb R^N)}&\lesssim \|(\Phi_1, \Phi_2, \Phi_3)\|_X^2+|\beta|\|\Phi_3\|_{L^2(\mathbb R^N)}\\
&\lesssim \|(\Phi_1, \Phi_2, \Phi_3)\|_X^2+|\beta|\|(\Phi_1, \Phi_2, \Phi_3)\|_X\end{aligned}$$ and hence we need $|\beta|$ sufficiently small in order to get a contraction.
\end{proof}
\section{The reduced problem}\label{teorema}
We want to find $\rho_{2, \e}$ and $\rho_{3, \e}$  such that $\mathfrak c_2=\mathfrak c_3=0$. In order to do so, we have to consider the second and the third equation in \eqref{DSnonlin} and we have to multiply by $\mathcal Z_\e$ and $\mathcal Y_\e$ respectively. \\ We fix $(\Phi_1, \Phi_2, \Phi_3)\in\mathcal K^\bot$  given by Proposition \ref{nonlin}.
\begin{lemma}\label{prolinN}
It holds
$$\left|\left(\mathcal N_2(\Phi_1, \Phi_2,  \Phi_3)-\mathcal L_2(\Phi_1, \Phi_2,  \Phi_3), \mathcal Z_\e\right)_{L^2(\RN)}\right|\lesssim |\beta|\tau_\e $$ and 
$$\left|\left(\mathcal N_3(\Phi_1, \Phi_2,  \Phi_3)-\mathcal L_3(\Phi_1, \Phi_2,  \Phi_3), \mathcal Y_\e\right)_{L^2(\RN)}\right|\lesssim |\beta|\tau_\e$$
\end{lemma}
\begin{proof}
We will prove only the first result (reasoning similarly we can obtain the second).\\
By using (\ref{lin2}) and (\ref{Ze}) we get that
\begin{equation*}
    \begin{aligned}
        \int_{\RN} \mathcal L_2(\Phi_1, \Phi_2, \Phi_3) \mathcal Z_\e \, dx &= \int_{\RN} \nabla \Phi_2 \nabla \mathcal Z_\e +W_2(\varepsilon x) \Phi_2 \mathcal Z_\e \, dx -3\mu_2 \int_{\RN} (\Uae+\varphi_2)^2 \mathcal Z_\e \Phi_2\, dx \\
        &- a_{21} \int_{\RN}(\Uae+\varphi_2)\Phi_1(\e x) \mathcal Z_\e \,dx -\beta a_{23}\int_{\RN} (\Ube+\varphi_3)\Phi_2 \mathcal Z_\e \,dx\\ &-a_{21}\int_{\RN} (\Upsilon+\varphi_1)(\e x)\Phi_2 \mathcal Z_\e\, dx \\
        &=\int_{\RN}(W_2(\e x)-\omega_2)\Phi_2 \mathcal Z_\e \, dx \\
        &+3\mu_2\int_{\RN} \left((\Uap)^2\partial_1 \Uap-(\Uam)^2\partial_1 \Uam\right)\Phi_2\, dx\\
        &-3\mu_2 \int_{\RN} (\Uae+\varphi_2)^2 \mathcal Z_\e \Phi_2\, dx - a_{21} \int_{\RN}(\Uae+\varphi_2)\Phi_1(\e x) \mathcal Z_\e \,dx\\
        & -\beta a_{23}\int_{\RN} (\Ube+\varphi_3)\Phi_2 \mathcal Z_\e \,dx-a_{21}\int_{\RN} (\Upsilon+\varphi_1)(\e x)\Phi_2 \mathcal Z_\e\, dx \\
        &= \int_{\RN}(\omega_2(\e x)-\omega_2)\Phi_2 \mathcal Z_\e \, dx+3\mu_2 \int_{\RN} (\Uap)^2\partial_1 \Uam \Phi_2 \, dx\\
        &-3\mu_2\int_{\RN}(\Uam)^2\partial_1 \Uap \Phi_2 \, dx -6\mu_2 \int_{\RN} \Uap \Uam \mathcal Z_\e \Phi_2 \, dx\\
        &-3\mu_2 \int_{\RN}\varphi_2^2 \Phi_2 \mathcal Z_\e \, dx -6\mu_2\int_{\RN}\Uae \varphi_2 \Phi_2 \mathcal Z_\e \, dx
        -a_{21}\int_{\RN}\varphi_1(\e x)\Phi_2 \mathcal Z_\e \, dx\\
        &-a_{21}\int_{\RN}\Uae\Phi_1(\e x) \mathcal Z_\e \, dx-a_{21}\int_{\RN}\varphi_2 \Phi_1(\e x) \mathcal Z_\e \, dx\\
        & -\beta a_{23}\int_{\RN} (\Ube+\varphi_3)\Phi_2 \mathcal Z_\e \,dx.
    \end{aligned}
\end{equation*}
Let us study the right hand side. First of all, we have that
$$\begin{aligned}\int_{\RN} (\omega_2(\e x )-\omega_2)\Phi_2 \mathcal Z_\e \, dx &\lesssim \left(\int_{\RN} |\omega_2(\e x )-\omega_2|^2 \mathcal Z_\e^2 \, dx\right)^{1/2} \|\Phi_2\|_{L^2(\RN)}\\ & \lesssim \rho^2_{2,\e}\|\Phi_2\|_{L^2(\RN}\lesssim \tau_\e^2.\end{aligned} $$
By applying Lemma \ref{ACR}, it results 
\begin{equation*}
    \begin{aligned}
        \int_{\RN} (\Uap)^2\partial_1 \Uam \Phi_2 \, dx&\lesssim \left(\int_{\RN} (\Uap)^4\Uam^2 \, dx\right)^{1/2} \|\Phi_2\|_{L^2(\RN)}  \lesssim \\
        &\lesssim e^{-2\sqrt{\omega_2}\frac{\rho_{2, \e}}{\e}}\left(\frac{\rho_{2, \e}}{\e}\right)^{-\frac{N-1}{2}} \|\Phi_2\|_{L^2(\RN)}\lesssim \tau_\e^2.
    \end{aligned}
\end{equation*}
Similarly, we have that
$$\int_{\RN}(\Uam)^2\partial_1 \Uap \Phi_2 \, dx\lesssim \tau_\e^2$$ and
$$\int_{\RN}\Uap \Uam\mathcal Z_\e \Phi_2 \, dx\lesssim \left(\int_{\RN}(\Uap)^4 (\Uam)^2 \, dx\right)^{1/2} \|\Phi_2\|_{L^2(\RN)}\lesssim \tau_\e^2. $$
In addition, Lemma \ref{hatvarphi}, \ref{esistphii} yield
$$\int_{\RN}\varphi_2^2 \Phi_2 \mathcal Z_\e \, dx\lesssim \|\varphi_2\|^2_{L^2(\RN)} \|\Phi_2\|_{L^2(\RN)}\lesssim \e^N \|\Phi_2\|_{L^2(\RN)}\lesssim \e^{N}\tau_\e.$$
$$\int_{\RN}\Uae \varphi_2 \Phi_2 \mathcal Z_\e \, dx \lesssim \|\varphi_2\|_{L^2(\RN)} \|\Phi_2\|_{L^2(\RN)}\lesssim \e^{\frac{N}{2}}\tau_\e$$ and
$$\int_{\RN}\varphi_1(\e x)\Phi_2 Z_\e \, dx\lesssim \|\varphi_1\|_{L^\infty(\RN)} \|\Phi_2\|_{L^2(\RN)}\lesssim \e^{\frac{N}{2}}\tau_\e.$$
Now
$$\int_{\RN}\Uae\Phi_1(\e x) \mathcal Z_\e \, dx=\int_{\RN}\Uae(\Phi_1(\e x)-\Phi_1(0)) \mathcal Z_\e \, dx+ \int_{\RN}\Uae\Phi_1(0) \mathcal Z_\e \, dx.$$
We remark that $\Phi_1\in H^2_V(\RN)$ and hence $\Phi_1\in C^{0, \frac 12}(\mathbb R^3)$ or $\Phi_1\in C^{0, \alpha}(\mathbb R^2)$ for any $\alpha\in (0, 1)$. Hence
$$|\Phi_1(\e x)-\Phi_1(0)|\lesssim \|\Phi_1\|_{H^2_V} \e^{\frac 12}|x|^{\frac 12}\quad \mbox{in}\,\,\mathbb R^3$$ and 
$$|\Phi_1(\e x)-\Phi_1(0)|\lesssim \|\Phi_1\|_{H^2_V} \e^{\alpha}|x|^{\alpha}\quad \mbox{in}\,\,\mathbb R^2.$$ Then 
\begin{equation*}
    \begin{aligned}
        \int_{\RN}\Uae(\Phi_1(\e x)-\Phi_1(0)) \mathcal Z_\e \, dx&\lesssim\|\Phi_1\|_{L^2(\RN)}\left\{\begin{aligned} & \int_{\RN}\e^{\frac 12}|x|^{\frac{1}{2}} (\Uae)^2 \, dx\quad \mbox{if}\,\, N=3\\
& \int_{\RN}\e^{\alpha}|x|^{\alpha} (\Uae)^2 \, dx\quad \mbox{if}\,\, N=2\end{aligned}\right.\\
&\lesssim \|\Phi_1\|_{L^2(\RN)}\left\{\begin{aligned} & \int_{\RN}\e^{\frac 12}\left|x+\frac{P_{2,\e}}{\e}\right|^{\frac{1}{2}} (\Uae)^2 \, dx\quad \mbox{if}\,\, N=3\\
& \int_{\RN}\e^{\alpha}\left|x+\frac{P_{2,\e}}{\e}\right|^{\alpha} (\Uae)^2 \, dx\quad \mbox{if}\,\, N=2\end{aligned}\right.\\
&\lesssim \|\Phi_1\|_{L^2(\RN)}\left\{\begin{aligned} & \rho^{\frac{1}{2}}_{2,\e}\quad \mbox{if}\,\, N=3\\
&  \rho^{\alpha}_{2,\e}  \quad \mbox{if}\,\, N=2  \end{aligned}\right.\\
&\lesssim \left\{\begin{aligned} & \e^{\frac 1 2 -\sigma}\tau_\e\quad &\mbox{if}\,\, N=3\\
&  \e^{\alpha-\sigma}\tau_\e  \quad &\mbox{if}\,\, N=2  \end{aligned}\right.\\
\end{aligned}
\end{equation*}
and, applying Lemma \ref{ACR} and the fact that $\int_{\RN} \mathcal U_{2, \pm P_2}\partial_1 \mathcal U_{2, \pm P_2}=0$
\begin{equation*}
    \begin{aligned}
        \int_{\RN}\Uae\Phi_1(0) \mathcal Z_\e \, dx&=-\Phi_1(0)\int_{\RN}(\Uap\partial_1\Uam-\Uam\partial_1\Uap)\, dx\\
        &\lesssim \|\Phi_1\|_{L^\infty(\RN)}\int_{\RN}|\Uap\partial_1\Uam-\Uam\partial_1\Uap|\, dx\\
        &\lesssim  \|\Phi_1\|_{H^2_V(\RN)} e^{-2\sqrt{\omega_2}\frac{\rho_{2, \e}}{\e}}\left(\frac{\rho_{2, \e}}{\e}\right)^{-\frac{N-1}{2}}\lesssim \tau_\e^2.
    \end{aligned}
\end{equation*}
Moreover
$$\left|\int_{\RN}\varphi_2 \Phi_1(\e x) \mathcal Z_\e \, dx\right|\lesssim \|\varphi_2\|_{L^2(\RN)} \|\Phi_1\|_{L^\infty(\RN)}\lesssim\e^{\frac{N}{2}}\|\Phi_1\|_{H^2_V(\RN)}\lesssim \e^{\frac N 2}\tau_\e.$$
At the end, reasoning as in Lemma \ref{errore}, we have $$\int_{\RN}(\mathcal U_{3, \e}+\varphi_3)\mathcal Z_\e \Phi_2\lesssim \|\Phi_2\|_{L^2(\RN)}\|\mathcal U_{3, \e}\mathcal U_{2, \e}\|_{L^2(\RN)}\lesssim\eta_\e\tau_\e.$$
Moreover
$$ \begin{aligned}
        \int_{\RN} \mathcal N_2(\Phi_1, \Phi_2, \Phi_3) \mathcal Z_\e \, dx &\lesssim \|(\Phi_1, \Phi_2, \Phi_3\|_X^2+|\beta|\int_{\RN}(\mathcal U_{2, \e}+\varphi_2)\mathcal Z_\e \Phi_3\\ &\lesssim \tau_\e^2+|\beta\|\Phi_3\|_{L^2(\RN)}\lesssim |\beta|\tau_\e.\end{aligned}$$\end{proof}
Now we have to compute the projection of the error $\mathcal E_2$ and $\mathcal E_3$.\\ To simplify the notation we let $$\mathfrak F(\rho_{2, \e}, \rho_{3, \e}):=\left\{\begin{aligned}
 &2\frac{\rho_{2, \e}}{|\rho_\e|}e^{-\sqrt{\omega_3}\frac{|\rho_\e|}{\e}}\left(\frac{|\rho_\e|}{\e}\right)^{-\frac{N-1}{2}},\quad &\mbox{if}\,\, \omega_3<4\omega_2\\
 &2\frac{\rho_{2, \e}}{|\rho_\e|}e^{-2\sqrt{\omega_2}\frac{|\rho_\e|}{\e}}\left(\frac{|\rho_\e|}{\e}\right)^{-(N-1)},\quad &\mbox{if}\,\, \omega_3>4\omega_2\\
&2\frac{\rho_{2,\e}}{|\rho_\e|}e^{-\sqrt{\omega_3}\frac{|\rho_\e|}{\e}},\quad &\mbox{if}\,\, \omega_3=4\omega_2,\,\mbox{and}\, N=2\\
&2\frac{\rho_{2, \e}}{|\rho_\e|}e^{-\sqrt{\omega_3}\frac{|\rho_\e|}{\e}}\left(\frac{|\rho_\e|}{\e}\right)^{-1}\ln\frac{|\rho_\e|}{\e},\quad &\mbox{if}\,\, \omega_3=4\omega_2,\,\mbox{and}\, N=3\end{aligned}\right.$$
and
$$\mathfrak G(\rho_{2, \e}, \rho_{3, \e}):=\left\{\begin{aligned}
 &2\frac{\rho_{3, \e}}{|\rho_\e|}e^{-\sqrt{\omega_2}\frac{|\rho_\e|}{\e}}\left(\frac{|\rho_\e|}{\e}\right)^{-\frac{N-1}{2}},\quad &\mbox{if}\,\, \omega_2<4\omega_3\\
 &2\frac{\rho_{3, \e}}{|\rho_\e|}e^{-2\sqrt{\omega_3}\frac{|\rho_\e|}{\e}}\left(\frac{|\rho_\e|}{\e}\right)^{-(N-1)},\quad &\mbox{if}\,\, \omega_2>4\omega_3\\
&2\frac{\rho_{3,\e}}{|\rho_\e|}e^{-\sqrt{\omega_2}\frac{|\rho_\e|}{\e}},\quad &\mbox{if}\,\, \omega_2=4\omega_3,\,\mbox{and}\, N=2\\
&2\frac{\rho_{3, \e}}{|\rho_\e|}e^{-\sqrt{\omega_2}\frac{|\rho_\e|}{\e}}\left(\frac{|\rho_\e|}{\e}\right)^{-1}\ln\frac{|\rho_\e|}{\e},\quad &\mbox{if}\,\, \omega_2=4\omega_3,\,\mbox{and}\, N=3\end{aligned}\right.$$

\begin{lemma}\label{proiezioni}
It holds
\begin{equation}\label{pro1}\int_{\RN}\mathcal E_2 \mathcal Z_\e\, dx=\left[-\partial^2_{11}\omega_2(0)\mathfrak b \e \rho_{2, \e}-2\mu_2 \mathfrak c e^{-2\sqrt{\omega_2}\frac{\rho_{2, \e}}{\e}}\left(\frac{\rho_{2, \e}}{\e}\right)^{-\frac{N-1}{2}}-\beta a_{23}\mathfrak d\mathfrak F(\rho_{2, \e}, \rho_{3, \e})\right](1+o(1))\end{equation} for some positive $\mathfrak b$, $\mathfrak c$ and $\mathfrak d$.
and
\begin{equation}\label{pro2}\int_{\RN}\mathcal E_3 \mathcal Y_\e\, dx=\left[-\partial_{22}^2\omega_3(0)\bar{\mathfrak  b}\e\rho_{3, \e}-2\mu_3\bar{ \mathfrak c }e^{-2\sqrt{\omega_3}\frac{\rho_{3, \e}}{\e}}\left(\frac{\rho_{3, \e}}{\e}\right)^{-\frac{N-1}{2}}-\beta a_{32}\bar{\mathfrak d}\mathfrak G(\rho_{2, \e}, \rho_{3, \e})\right](1+o(1))\end{equation} for some positive $\bar{\mathfrak b}$, $\bar{\mathfrak c}$ and $\bar{\mathfrak d}$.
\end{lemma}
\begin{proof}
By recalling the definition of $\mathcal E_2$ we get
$$\begin{aligned}\int_{\RN}\mathcal E_2 \mathcal Z_\e\, dx&:=3\mu_2 \int_{\RN}\left((\Uap)^2\Uam+(\Uam)^2\Uap\right)\mathcal Z_\e\, dx\\
&+6\mu_2 \int_{\RN}\Uap\Uam\varphi_2\mathcal Z_\e\, dx+3\mu_2\int_{\RN}\Uae\varphi_2^2\mathcal Z_\e\, dx+\mu_2\int_{\RN}\varphi_2^3\mathcal Z_\e\, dx\\
&+\int_{\RN}(\omega_2-\omega_2(\e x))(\Uae+\varphi_2)\mathcal Z_\e\, dx+a_{21}\int_{\RN}\Uae(\varphi_1(\e x)-\varphi_1(0))\mathcal Z_\e\, dx\\&+a_{21}\int_{\RN}\varphi_1(\e x)\varphi_2\mathcal Z_\e\, dx+\beta a_{23}\int_{\RN}(\Uae+\varphi_2)(\Ube+\varphi_3)\mathcal Z_\e\, dx.\end{aligned}$$
As done in Lemma 4.2 of \cite{PPVV} we have that
$$\int_{\RN}(\omega_2-\omega_2(\e x))(\Uae+\varphi_2)\mathcal Z_\e\, dx =-\partial^2_{11}\omega_2(0)\mathfrak b \e \rho_\e+o(\e\rho_\e)$$ where $$\mathfrak b:=-\int_{\RN}y_1^2 \Ua(y)\frac{\Ua'(y)}{|y|}\, dy>0.$$ While
$$3\mu_2 \int_{\RN}\left((\Uap)^2\Uam+(\Uam)^2\Uap\right)\mathcal Z_\e\, dx=-2\mu_2 \mathfrak c e^{-2\sqrt{\omega_2}\frac{\rho_{2, \e}}{\e}}\left(\frac{\rho_{2, \e}}{\e}\right)^{-\frac{N-1}{2}}(1+o(1)).$$
Moreover $$\left|\int_{\RN}\Uae(\varphi_1(\e x)-\varphi_1(0))\mathcal Z_\e\, dx\right|\lesssim\int_{\RN} |\varphi_1(\e x)-\varphi_1(0)|(\Uap)^2\, dx \lesssim \e^2\int_{\RN}|x|^{2-\frac N m}(\Uap)^2\, dx=o(\e \rho_\e^2)$$ and
$$
\int_{\RN}\Uap\Uam\varphi_2\mathcal Z_\e\, dx=o\left(e^{-2\sqrt{\omega_2}\frac{\rho_{2, \e}}{\e}}\left(\frac{\rho_{2, \e}}{\e}\right)^{-\frac{N-1}{2}}\right).$$ Finally,
$$\int_{\RN}\varphi_2\varphi_1(\e x)\mathcal Z_\e\lesssim \e^N,\quad \int_{\RN}
 \Uae\varphi_2^2\mathcal Z_\e\lesssim \e^N,\quad \int_{\RN}\varphi_2^3\mathcal Z_\e\lesssim \e^N.$$ 
 We need to estimate the interaction between the second and the third equation.
 $$\begin{aligned}\int_{\RN}\Uae\Ube\mathcal Z_\e\, dx&=(1+o(1))\left(\int_{\RN}\Uap\Ubp \partial_1 \Uap\, dx +\int_{\RN}\Uap\Ubm\partial_1\Uap\, dx\right)\\
 &-(1+o(1))\left(\int_{\RN}\Uam\Ubp \partial_1 \Uam\, dx +\int_{\RN}\Uam\Ubm\partial_1\Uam\, dx\right)\\
 &=\frac 12 (1+o(1))\int_{\RN}\partial_1 (\Ua)^2(y) \Ub\left(y+\frac{P_3-P_2}{\e}\right)\, dy\\
 &+\frac 12 (1+o(1))\int_{\RN}\partial_1 (\Ua)^2(y) \Ub\left(y-\frac{P_3+P_2}{\e}\right)\, dy\\
 &-\frac 12 (1+o(1))\int_{\RN}\partial_1 (\Ua)^2(y) \Ub\left(y+\frac{P_3+P_2}{\e}\right)\, dy\\
 &-\frac 12(1+o(1))\int_{\RN}\partial_1 (\Ua)^2(y) \Ub\left(y+\frac{P_2-P_3}{\e}\right)\, dy
 \end{aligned}$$
 We have that
 $$(\Ua)^2\sim e^{-2\sqrt{\omega_2}|x|}|x|^{-(N-1)},\quad \Ub\sim e^{-\sqrt{\omega_3}|x|}|x|^{-\frac{N-1}{2}}$$
 then, letting $\zeta:=\frac{P_3-P_2}{\e}$, we get by Remark \ref{ossdecay} that
 $$\int_{\RN}\partial_1 (\Ua)^2(y) \Ub\left(y+\zeta\right)\, dy\sim\left\{\begin{aligned}
 &\frac{\zeta_1}{|\zeta|}e^{-\sqrt{\omega_3}|\zeta|}|\zeta|^{-\frac{N-1}{2}},\quad &\mbox{if}\,\, \omega_3<4\omega_2\\
 &\frac{\zeta_1}{|\zeta|}e^{-2\sqrt{\omega_2}|\zeta|}|\zeta|^{-(N-1)},\quad &\mbox{if}\,\, \omega_3>4\omega_2\\
&\frac{\zeta_1}{|\zeta|}e^{-\sqrt{\omega_3}|\zeta|},\quad &\mbox{if}\,\, \omega_3=4\omega_2,\,\mbox{and}\, N=2\\
&\frac{\zeta_1}{|\zeta|}e^{-\sqrt{\omega_3}|\zeta|}|\zeta|^{-1}\ln|\zeta|,\quad &\mbox{if}\,\, \omega_3=4\omega_2,\,\mbox{and}\, N=3\end{aligned}\right.$$
Moreover
$$\int_{\RN}\partial_1 (\Ua)^2(y) \Ub\left(y-\zeta\right)\, dy\sim\left\{\begin{aligned}
 &-\frac{\zeta_1}{|\zeta|}e^{-\sqrt{\omega_3}|\zeta|}|\zeta|^{-\frac{N-1}{2}},\quad &\mbox{if}\,\, \omega_3<4\omega_2\\
 &-\frac{\zeta_1}{|\zeta|}e^{-2\sqrt{\omega_2}|\zeta|}|\zeta|^{-(N-1)},\quad &\mbox{if}\,\, \omega_3>4\omega_2\\
&-\frac{\zeta_1}{|\zeta|}e^{-\sqrt{\omega_3}|\zeta|},\quad &\mbox{if}\,\, \omega_3=4\omega_2,\,\mbox{and}\, N=2\\
&-\frac{\zeta_1}{|\zeta|}e^{-\sqrt{\omega_3}|\zeta|}|\zeta|^{-1}\ln|\zeta|,\quad &\mbox{if}\,\, \omega_3=4\omega_2,\,\mbox{and}\, N=3\end{aligned}\right.$$
Lettin $\xi:=\frac{P_3+P_2}{\e}$ then
$$\int_{\RN}\partial_1 (\Ua)^2(y) \Ub\left(y-\xi\right)\, dy\sim\left\{\begin{aligned}
 &-\frac{\xi_1}{|\xi|}e^{-\sqrt{\omega_3}|\xi|}|\xi|^{-\frac{N-1}{2}},\quad &\mbox{if}\,\, \omega_3<4\omega_2\\
 &-\frac{\xi_1}{|\xi|}e^{-2\sqrt{\omega_2}|\xi|}|\xi|^{-(N-1)},\quad &\mbox{if}\,\, \omega_3>4\omega_2\\
&-\frac{\xi_1}{|\xi|}e^{-\sqrt{\omega_3}|\xi|},\quad &\mbox{if}\,\, \omega_3=4\omega_2,\,\mbox{and}\, N=2\\
&-\frac{\xi_1}{|\xi|}e^{-\sqrt{\omega_3}|\xi|}|\xi|^{-1}\ln|\xi|,\quad &\mbox{if}\,\, \omega_3=4\omega_2,\,\mbox{and}\, N=3\end{aligned}\right.$$
and finally
$$\int_{\RN}\partial_1 (\Ua)^2(y) \Ub\left(y+\xi\right)\, dy\sim\left\{\begin{aligned}
 &\frac{\xi_1}{|\xi|}e^{-\sqrt{\omega_3}|\xi|}|\xi|^{-\frac{N-1}{2}},\quad &\mbox{if}\,\, \omega_3<4\omega_2\\
 &\frac{\xi_1}{|\xi|}e^{-2\sqrt{\omega_2}|\xi|}|\xi|^{-(N-1)},\quad &\mbox{if}\,\, \omega_3>4\omega_2\\
&\frac{\xi_1}{|\xi|}e^{-\sqrt{\omega_3}|\xi|},\quad &\mbox{if}\,\, \omega_3=4\omega_2,\,\mbox{and}\, N=2\\
&\frac{\xi_1}{|\xi|}e^{-\sqrt{\omega_3}|\xi|}|\xi|^{-1}\ln|\xi|,\quad &\mbox{if}\,\, \omega_3=4\omega_2,\,\mbox{and}\, N=3\end{aligned}\right.$$
Now we remark that 
$$\zeta:=\frac{P_3-P_2}{\e}:=\frac 1 \e (-\rho_{2, \e}, \rho_{3, \e}, 0),\quad \xi:=\frac{P_3+P_2}{\e}:=\frac 1 \e (\rho_{2, \e}, \rho_{3, \e}, 0)$$ and hence $|\zeta|:=|\xi|:=\frac 1 \e \sqrt{(\rho_{2, \e})^2+(\rho_{3, \e})^2}:=\frac 1 \e|\rho_\e|.$\\ 
Hence
$$\int_{\RN}\Uae\Ube\mathcal Z_\e\, dx:=\mathfrak d(1+o(1))\left\{\begin{aligned}
 &-2\frac{\rho_{2, \e}}{|\rho_\e|}e^{-\sqrt{\omega_3}\frac{|\rho_\e|}{\e}}\left(\frac{|\rho_\e|}{\e}\right)^{-\frac{N-1}{2}},\quad &\mbox{if}\,\, \omega_3<4\omega_2\\
 &-2\frac{\rho_{2, \e}}{|\rho_\e|}e^{-2\sqrt{\omega_2}\frac{|\rho_\e|}{\e}}\left(\frac{|\rho_\e|}{\e}\right)^{-(N-1)},\quad &\mbox{if}\,\, \omega_3>4\omega_2\\
&-2\frac{\rho_{2,\e}}{|\rho_\e|}e^{-\sqrt{\omega_3}\frac{|\rho_\e|}{\e}},\quad &\mbox{if}\,\, \omega_3=4\omega_2,\,\mbox{and}\, N=2\\
&-2\frac{\rho_{2, \e}}{|\rho_\e|}e^{-\sqrt{\omega_3}\frac{|\rho_\e|}{\e}}\left(\frac{|\rho_\e|}{\e}\right)^{-1}\ln\frac{|\rho_\e|}{\e},\quad &\mbox{if}\,\, \omega_3=4\omega_2,\,\mbox{and}\, N=3\end{aligned}\right.$$
for some positive constant $\mathfrak d$.\\ By using Remark \ref{decay} it is easy to see that $$\int_{\RN} \Uae\varphi_3\mathcal Z_\e =o\left(\mathfrak F(\rho_{2, \e}, \rho_{3, \e})\right),\quad \int_{\RN} \Ube\varphi_2\mathcal Z_\e=o\left(\mathfrak F(\rho_{2, \e}, \rho_{3, \e})\right)\quad \int_{\RN}\varphi_2\varphi_3\mathcal Z_\e=o\left(\mathfrak F(\rho_{2, \e}, \rho_{3, \e})\right).$$
Putting together the various estimates we get \eqref{pro1}.\\ For what concerning the projection of $\mathcal E_3$ we have that
$$\begin{aligned}\int_{\RN}\mathcal E_3\mathcal Y_\e\, dx &:=
3\mu_3 \int_{\RN}\left((\Ubp)^2\Ubm+(\Ubm)^2\Ubp\right)\mathcal Y_\e\, dx\\
&+6\mu_2 \int_{\RN}\Ubp\Ubm\varphi_3\mathcal Y_\e\, dx+3\mu_3\int_{\RN}\Ube\varphi_3^2\mathcal Y_\e\, dx+\mu_3\int_{\RN}\varphi_3^3\mathcal Y_\e\, dx\\
&+\int_{\RN}(\omega_3-\omega_3(\e x))(\Ube+\varphi_3)\mathcal Y_\e\, dx+a_{31}\int_{\RN}\Ube(\varphi_1(\e x)-\varphi_1(0))\mathcal Y_\e\, dx\\
&+a_{31}\int_{\RN}\varphi_1(\e x)\varphi_3\mathcal Y_\e\, dx+\beta a_{32}\int_{\RN}(\Uae+\varphi_2)(\Ube+\varphi_3)\mathcal Y_\e\, dx.\end{aligned}$$
As before
$$\int_{\RN}(\omega_3-\omega_3(\e x))(\Ube+\varphi_3)\mathcal Y_\e\, dx=-\partial_{22}^2\omega_3(0)\bar{\mathfrak  b}\e\rho_{3, \e}+o(\e\rho_{3, \e} )$$ where
$$\bar{\mathfrak b}:=-\int_{\RN}y_2^2 \Ub(y)\frac{\Ub'(y)}{|y|}\, dy>0.$$
While
$$3\mu_3 \int_{\RN}\left((\Ubp)^2\Ubm+(\Ubm)^2\Ubp\right)\mathcal Y_\e\, dx=-2\mu_3 \bar{\mathfrak c} e^{-2\sqrt{\omega_3}\frac{\rho_{3, \e}}{\e}}\left(\frac{\rho_{3, \e}}{\e}\right)^{-\frac{N-1}{2}}(1+o(1)).$$
Again
$$\int_{\RN}\Ube(\varphi_1(\e x)-\varphi_1(0))\mathcal Y_\e\, dx=o(\e \rho_{3, \e}^2)$$ and
$$
\int_{\RN}\Ubp\Ubm\varphi_3\mathcal Y_\e\, dx=o\left(e^{-2\sqrt{\omega_3}\frac{\rho_{3, \e}}{\e}}\left(\frac{\rho_{3, \e}}{\e}\right)^{-\frac{N-1}{2}}\right).$$ Finally,
$$\int_{\RN}\varphi_3\varphi_1(\e x)\mathcal Y_\e\lesssim \e^N,\quad \int_{\RN}
 \Ube\varphi_3^2\mathcal Y_\e\lesssim \e^N,\quad \int_{\RN}\varphi_3^3\mathcal Y_\e\lesssim \e^N.$$ 
 We need to estimate the interaction between the second and the third equation.
 $$\begin{aligned}\int_{\RN}\Uae\Ube\mathcal Y_\e\, dx&=(1+o(1))\left(\int_{\RN}\Uap\Ubp \partial_2 \Ubp\, dx -\int_{\RN}\Uap\Ubm\partial_2\Ubm\, dx\right)\\
 &+(1+o(1))\left(\int_{\RN}\Uam\Ubp \partial_2 \Ubp\, dx -\int_{\RN}\Uam\Ubm\partial_2\Ubm\, dx\right)\\
 &=\frac 12 (1+o(1))\int_{\RN}\partial_2 (\Ub)^2(y) \Ua\left(y-\frac{P_3-P_2}{\e}\right)\, dy\\
 &-\frac 12 (1+o(1))\int_{\RN}\partial_2 (\Ub)^2(y) \Ua\left(y+\frac{P_3+P_2}{\e}\right)\, dy\\
 &+\frac 12 (1+o(1))\int_{\RN}\partial_2 (\Ub)^2(y) \Ua\left(y-\frac{P_3+P_2}{\e}\right)\, dy\\
 &-\frac 12(1+o(1))\int_{\RN}\partial_2 (\Ub)^2(y) \Ua\left(y+\frac{P_3-P_2}{\e}\right)\, dy
 \end{aligned}$$
and recalling the definition of $\zeta$ and $\xi$ we get
$$\int_{\RN}\partial_2 (\Ub)^2(y) \Ua\left(y-\zeta\right)\, dy\sim\left\{\begin{aligned}
 &-\frac{\zeta_2}{|\zeta|}e^{-\sqrt{\omega_2}|\zeta|}|\zeta|^{-\frac{N-1}{2}},\quad &\mbox{if}\,\, \omega_2<4\omega_3\\
 &-\frac{\zeta_2}{|\zeta|}e^{-2\sqrt{\omega_3}|\zeta|}|\zeta|^{-(N-1)},\quad &\mbox{if}\,\, \omega_2>4\omega_3\\
&-\frac{\zeta_2}{|\zeta|}e^{-\sqrt{\omega_2}|\zeta|},\quad &\mbox{if}\,\, \omega_2=4\omega_3,\,\mbox{and}\, N=2\\
&-\frac{\zeta_2}{|\zeta|}e^{-\sqrt{\omega_2}|\zeta|}|\zeta|^{-1}\ln|\zeta|,\quad &\mbox{if}\,\, \omega_2=4\omega_3,\,\mbox{and}\, N=3\end{aligned}\right.$$

$$\int_{\RN}\partial_2 (\Ub)^2(y) \Ua\left(y+\xi\right)\, dy\sim\left\{\begin{aligned}
 &\frac{\xi_2}{|\xi|}e^{-\sqrt{\omega_2}|\xi|}|\xi|^{-\frac{N-1}{2}},\quad &\mbox{if}\,\, \omega_2<4\omega_3\\
 &\frac{\xi_2}{|\xi|}e^{-2\sqrt{\omega_3}|\xi|}|\xi|^{-(N-1)},\quad &\mbox{if}\,\, \omega_2>4\omega_3\\
&\frac{\xi_2}{|\xi|}e^{-\sqrt{\omega_2}|\xi|},\quad &\mbox{if}\,\, \omega_2=4\omega_3,\,\mbox{and}\, N=2\\
&\frac{\xi_2}{|\xi|}e^{-\sqrt{\omega_2}|\xi|}|\xi|^{-1}\ln|\xi|,\quad &\mbox{if}\,\, \omega_2=4\omega_3,\,\mbox{and}\, N=3\end{aligned}\right.$$

$$\int_{\RN}\partial_2 (\Ub)^2(y) \Ua\left(y-\xi\right)\, dy\sim\left\{\begin{aligned}
 &-\frac{\xi_2}{|\xi|}e^{-\sqrt{\omega_2}|\xi|}|\xi|^{-\frac{N-1}{2}},\quad &\mbox{if}\,\, \omega_2<4\omega_3\\
 &-\frac{\xi_2}{|\xi|}e^{-2\sqrt{\omega_3}|\xi|}|\xi|^{-(N-1)},\quad &\mbox{if}\,\, \omega_2>4\omega_3\\
&-\frac{\xi_2}{|\xi|}e^{-\sqrt{\omega_2}|\xi|},\quad &\mbox{if}\,\, \omega_2=4\omega_3,\,\mbox{and}\, N=2\\
&-\frac{\xi_2}{|\xi|}e^{-\sqrt{\omega_2}|\xi|}|\xi|^{-1}\ln|\xi|,\quad &\mbox{if}\,\, \omega_2=4\omega_3,\,\mbox{and}\, N=3\end{aligned}\right.$$
and
$$\int_{\RN}\partial_2 (\Ub)^2(y) \Ua\left(y+\zeta\right)\, dy\sim\left\{\begin{aligned}
 &\frac{\zeta_2}{|\zeta|}e^{-\sqrt{\omega_2}|\zeta|}|\zeta|^{-\frac{N-1}{2}},\quad &\mbox{if}\,\, \omega_2<4\omega_3\\
 &\frac{\zeta_2}{|\zeta|}e^{-2\sqrt{\omega_3}|\zeta|}|\zeta|^{-(N-1)},\quad &\mbox{if}\,\, \omega_2>4\omega_3\\
&\frac{\zeta_2}{|\zeta|}e^{-\sqrt{\omega_2}|\zeta|},\quad &\mbox{if}\,\, \omega_2=4\omega_3,\,\mbox{and}\, N=2\\
&\frac{\zeta_2}{|\zeta|}e^{-\sqrt{\omega_2}|\zeta|}|\zeta|^{-1}\ln|\zeta|,\quad &\mbox{if}\,\, \omega_2=4\omega_3,\,\mbox{and}\, N=3\end{aligned}\right.$$
Then
$$\int_{\RN}\Uae\Ube\mathcal Y_\e\, dx=\bar{\mathfrak d}(1+o(1))\left\{\begin{aligned}
 &-2\frac{\rho_{3, \e}}{|\rho_\e|}e^{-\sqrt{\omega_2}\frac{|\rho_\e|}{\e}}\left(\frac{|\rho_\e|}{\e}\right)^{-\frac{N-1}{2}},\quad &\mbox{if}\,\, \omega_2<4\omega_3\\
 &-2\frac{\rho_{3, \e}}{|\rho_\e|}e^{-2\sqrt{\omega_3}\frac{|\rho_\e|}{\e}}\left(\frac{|\rho_\e|}{\e}\right)^{-(N-1)},\quad &\mbox{if}\,\, \omega_2>4\omega_3\\
&-2\frac{\rho_{3,\e}}{|\rho_\e|}e^{-\sqrt{\omega_2}\frac{|\rho_\e|}{\e}},\quad &\mbox{if}\,\, \omega_2=4\omega_3,\,\mbox{and}\, N=2\\
&-2\frac{\rho_{3, \e}}{|\rho_\e|}e^{-\sqrt{\omega_2}\frac{|\rho_\e|}{\e}}\left(\frac{|\rho_\e|}{\e}\right)^{-1}\ln\frac{|\rho_\e|}{\e},\quad &\mbox{if}\,\, \omega_2=4\omega_3,\,\mbox{and}\, N=3\end{aligned}\right.$$
By using Remark \ref{decay} it is easy to see that $$\int_{\RN} \Ube\varphi_2\mathcal Y_\e =o\left(\mathfrak G(\rho_{2, \e}, \rho_{3, \e})\right),\quad \int_{\RN} \Uae\varphi_3\mathcal Y_\e=o\left(\mathfrak G(\rho_{2, \e}, \rho_{3, \e})\right)\quad \int_{\RN}\varphi_2\varphi_3\mathcal Y_\e=o\left(\mathfrak G(\rho_{2, \e}, \rho_{3, \e})\right).$$
Putting together all the terms \eqref{pro2} follows.
\end{proof}

\begin{proof}[Proof of Theorem \ref{thm1}]
Let us define 
$$\mathfrak H_1(\rho_{2, \e}, \rho_{3, \e}):=-\partial^2_{11}\omega_2(0)\mathfrak b \e \rho_{2, \e}-2\mu_2 \mathfrak c e^{-2\sqrt{\omega_2}\frac{\rho_{2, \e}}{\e}}\left(\frac{\rho_{2, \e}}{\e}\right)^{-\frac{N-1}{2}}-\beta a_{23}\mathfrak d\mathfrak F(\rho_{2, \e}, \rho_{3, \e})$$
and
$$\mathfrak H_2(\rho_{2, \e}, \rho_{3, \e}):=-\partial_{22}^2\omega_3(0)\bar{\mathfrak  b}\e\rho_{3, \e}-2\mu_3\bar{ \mathfrak c }e^{-2\sqrt{\omega_3}\frac{\rho_{3, \e}}{\e}}\left(\frac{\rho_{3, \e}}{\e}\right)^{-\frac{N-1}{2}}-\beta a_{32}\bar{\mathfrak d}\mathfrak G(\rho_{2, \e}, \rho_{3, \e}).$$
We have to find $\rho_{2,\e}$ and $\rho_{3,\e}$ such that  $\mathfrak c_2=0$ and $\mathfrak c_3=0$. Then, putting together Lemmas \ref{prolinN} and \ref{proiezioni} we get that
\begin{equation}\label{sist}
\left\{\begin{aligned} \mathfrak H_1(\rho_{2, \e}, \rho_{3, \e})(1+o(1))+\mathcal O(|\beta|\tau_\e)&=0\\ \mathfrak H_2(\rho_{2, \e}, \rho_{3, \e})(1+o(1))+\mathcal O(|\beta|\tau_\e)&=0 \end{aligned}\right.\end{equation}
We want to understand better the functions $\mathfrak H_1$ and $\mathfrak H_2$. Here, we consider only the case $\omega_2\leq\frac{\omega_3}{4}$ since the case $\omega_3\leq\frac{\omega_2}{4}$ (and hence $\omega_2\geq 4\omega_3$) can be treated in a similar way.\\
If $\omega_2<\frac{\omega_3}{4}$ then
$$\mathfrak F(\rho_{2, \e}, \rho_{3, \e})=2\frac{\rho_{2, \e}}{|\rho_\e|}e^{-2\sqrt{\omega_2}\frac{|\rho_\e|}{\e}}\left(\frac{|\rho_\e|}{\e}\right)^{-(N-1)}$$ and if $\omega_2=\frac{\omega_3}{4}$ then
$$\mathfrak F(\rho_{2, \e}, \rho_{3, \e})=\left\{\begin{aligned}&2\frac{\rho_{2, \e}}{|\rho_\e|}e^{-\sqrt{\omega_3}\frac{|\rho_\e|}{\e}}\quad &\mbox{if}\, N=2\\
&2\frac{\rho_{2, \e}}{|\rho_\e|}e^{-\sqrt{\omega_3}\frac{|\rho_\e|}{\e}}\frac{\e}{|\rho_\e|}\ln\frac{|\rho_\e|}{\e}\quad &\mbox{if}\, N=3.\end{aligned}\right.$$
In any case $$\mathfrak F(\rho_{2, \e}, \rho_{3, \e})=o\left(e^{-2\sqrt{\omega_2}\frac{\rho_{2, \e}}{\e}}\left(\frac{\rho_{2, \e}}{\e}\right)^{-\frac{N-1}{2}}\right)$$ since $\rho_{2, \e}<|\rho_\e|$. Moreover in both cases
$$\mathfrak G(\rho_{2, \e}, \rho_{3, \e})=2\frac{\rho_{3, \e}}{|\rho_\e|}e^{-\sqrt{\omega_2}\frac{|\rho_\e|}{\e}}\left(\frac{|\rho_\e|}{\e}\right)^{-\frac{N-1}{2}}.$$
Hence the first equation of \eqref{sist} becomes 
\begin{equation}\label{case1}\underbrace{\left(-\partial^2_{11}\omega_2(0)\mathfrak b \e \rho_{2, \e}-2\mu_2 \mathfrak c e^{-2\sqrt{\omega_2}\frac{\rho_{2, \e}}{\e}}\left(\frac{\rho_{2, \e}}{\e}\right)^{-\frac{N-1}{2}}\right)}_{:=\mathfrak f(\rho_{2, \e})}(1+o(1))+\mathcal O(|\beta|\tau_\e)=0\end{equation} 
and the second equation of \eqref{sist} becomes 
\begin{equation}\label{case2}\begin{aligned}&\left(-\partial_{22}^2\omega_3(0)\bar{\mathfrak  b}\e\rho_{3, \e}-2\mu_3\bar{\mathfrak c} e^{-2\sqrt{\omega_3}\frac{\rho_{3, \e}}{\e}}\left(\frac{\rho_{3, \e}}{\e}\right)^{-\frac{N-1}{2}}-2\beta a_{32}\frac{\rho_{3, \e}}{|\rho_\e|}e^{-\sqrt{\omega_2}\frac{|\rho_\e|}{\e}}\left(\frac{|\rho_\e|}{\e}\right)^{-\frac{N-1}{2}}\right)(1+o(1))\\&+\mathcal O(|\beta|\tau_\e)=0.\end{aligned}\end{equation} where
 $-\partial^2_{11}\omega_2(0)\mathfrak b>0$ and $-\partial_{22}^2\omega_3(0)\bar{\mathfrak  b}>0$ with $a_{32}>0$ while  $-\partial_{22}^2\omega_3(0)\bar{\mathfrak  b}<0$ with $a_{32}<0$.\\
 Now we choose  $$\beta= \beta_0\e^b,\quad \beta_0>0$$ so that
 \beq\label{conprima}\beta\tau_\e =o(\e^2|\ln\e|)\eeq and \beq\label{conseconda}e^{-2\sqrt{\omega_3}\frac{\rho_{3, \e}}{\e}}\left(\frac{\rho_{3, \e}}{\e}\right)^{-\frac{N-1}{2}}=o\left(\beta a_{32}\frac{\rho_{3, \e}}{|\rho_\e|}e^{-\sqrt{\omega_2}\frac{|\rho_\e|}{\e}}\left(\frac{|\rho_\e|}{\e}\right)^{-\frac{N-1}{2}}\right).\eeq
 Hence the system that we have to solve reduces to
 $$\left\{\begin{aligned} &\mathfrak f(\rho_{2, \e})(1+o(1))=0\\ &
\mathfrak g(\rho_{2, \e}, \rho_{3, \e})(1+o(1))=0\end{aligned}\right.$$ where $$\mathfrak g(\rho_{2, \e}, \rho_{3, \e}):=\left(-\partial_{22}^2\omega_3(0)\bar{\mathfrak  b}\e\rho_{3, \e}-2\beta a_{32}\frac{\rho_{3, \e}}{|\rho_\e|}e^{-\sqrt{\omega_2}\frac{|\rho_\e|}{\e}}\left(\frac{|\rho_\e|}{\e}\right)^{-\frac{N-1}{2}}\right).$$
Then, it is easy to see that for $\e$ small 
$$\mathfrak f\left(\left(\frac{1}{\sqrt\omega_2}-\delta\right)\e|\ln\e|\right)<0,\quad \mathfrak f\left(\left(\frac{1}{\sqrt\omega_2}+\delta\right)\e|\ln\e|\right)>0$$ and hence has a zero at $$\bar\rho_2:=\left(\frac{1}{\sqrt\omega_2}+o(1)\right)\e|\ln\e|.$$ Moreover
 $$\mathfrak g\left(\bar\rho_2,\left(\frac{\sqrt {(1-b)(3-b)}}{\sqrt{\omega_2}}-\delta\right)\e|\ln\e|\right)<0,\quad \mathfrak g\left(\bar\rho_2,\left(\frac{(1-b)(3-b}{\sqrt{\omega_2}}+\delta\right)\e|\ln\e|\right)>0$$ if $-\partial_{22}^2\omega_3(0)\bar{\mathfrak  b}>0$ and $a_{32}>0$ while  
 $$\mathfrak g\left(\bar\rho_2,\left(\frac{\sqrt {(1-b)(3-b)}}{\sqrt{\omega_2}}-\delta\right)\e|\ln\e|\right)>0,\quad \mathfrak g\left(\bar\rho_2,\left(\frac{(1-b)(3-b}{\sqrt{\omega_2}}+\delta\right)\e|\ln\e|\right)<0$$ if $-\partial_{22}^2\omega_3(0)\bar{\mathfrak  b}<0$ and  $a_{32}<0$. In any case $\mathfrak g$ has a zero at some $$\bar{\rho}_3:=\left(\frac{\sqrt{(1-b)(3-b)}}{\sqrt\omega_2}+o(1)\right)\e|\ln\e|.$$
 It remains to show that \eqref{conprima} and \eqref{conseconda} hold.\\
 Indeed, $$\frac{\beta\tau_\e}{\e^2|\ln\e|}\lesssim \e^b|\ln\e|+\frac{\e^{b-\delta}}{|\ln\e|^{\frac{N-1}{2}+1}}+\frac{\e^{2\left(\sqrt{\frac{\omega_3}{\omega_2}}\sqrt{(1-b)(3-b)}-1\right)-\delta}}{|\ln\e|^{\frac{N-1}{2}+1}}\rightarrow 0$$ since $0<b<1$ and $\delta$ is sufficiently small.\\
 Moreover 
  $$\frac{e^{-2\sqrt{\omega_3}\frac{\rho_{3, \e}}{\e}}\left(\frac{\rho_{3, \e}}{\e}\right)^{-\frac{N-1}{2}}}{\beta e^{-\sqrt{\omega_2}\frac{|\rho_\e|}{\e}}\left(\frac{|\rho_\e|}{\e}\right)^{-\frac{N-1}{2}}}\lesssim \e^{2\left(\sqrt{\frac{\omega_3}{\omega_2}}\sqrt{(1-b)(3-b)}-1\right)-\delta}\rightarrow 0$$
since $b<1$ and $\delta$ is sufficiently small.\\
 The case $\omega_3\leq \frac{\omega_2}{4}$ can be treated in a similar way changing the role of the second and of the third equation.\\

Let now the case in which $\frac{\omega_3}{4}<\omega_2<4\omega_3$ and $\omega_2\neq \omega_3$. Then
we have that 
 $$\mathfrak F(\rho_{2, \e}, \rho_{3, \e})=2\frac{\rho_{2, \e}}{|\rho_\e|}e^{-\sqrt{\omega_3}\frac{|\rho_\e|}{\e}}\left(\frac{|\rho_\e|}{\e}\right)^{-\frac{N-1}{2}}$$  and 
 $$\mathfrak G(\rho_{2, \e}, \rho_{3, \e})=2\frac{\rho_{3, \e}}{|\rho_\e|}e^{-\sqrt{\omega_2}\frac{|\rho_\e|}{\e}}\left(\frac{|\rho_\e|}{\e}\right)^{-\frac{N-1}{2}}.$$
 Then the first equation becomes 
 \begin{equation}\label{case3}\begin{aligned}&\underbrace{\left(-\partial_{11}^2\omega_2(0){\mathfrak  b}\e\rho_{2, \e}-2\mu_2{\mathfrak c} e^{-2\sqrt{\omega_2}\frac{\rho_{2, \e}}{\e}}\left(\frac{\rho_{2, \e}}{\e}\right)^{-\frac{N-1}{2}}-2\beta a_{23}\frac{\rho_{2, \e}}{|\rho_\e|}e^{-\sqrt{\omega_3}\frac{|\rho_\e|}{\e}}\left(\frac{|\rho_\e|}{\e}\right)^{-\frac{N-1}{2}}\right)}_{:=\mathfrak f(\rho_{2, \e}, \rho_{3, \e})}(1+o(1))\\&+\mathcal O(|\beta|\tau_\e)=0.\end{aligned}\end{equation}
 and the second equation becomes
 \begin{equation}\label{case4}\begin{aligned}&\underbrace{\left(-\partial_{22}^2\omega_3(0)\bar{\mathfrak  b}\e\rho_{3, \e}-2\mu_3\bar{\mathfrak c} e^{-2\sqrt{\omega_3}\frac{\rho_{3, \e}}{\e}}\left(\frac{\rho_{3, \e}}{\e}\right)^{-\frac{N-1}{2}}-2\beta a_{32}\frac{\rho_{3, \e}}{|\rho_\e|}e^{-\sqrt{\omega_2}\frac{|\rho_\e|}{\e}}\left(\frac{|\rho_\e|}{\e}\right)^{-\frac{N-1}{2}}\right)}_{:=\mathfrak g(\rho_{2, \e}, \rho_{3, \e})}(1+o(1))\\ &+\mathcal O(|\beta|\tau_\e)=0.\end{aligned}\end{equation}
 In this case we choose 
 \begin{equation}\label{con443}|\beta|\lesssim \beta_\e:=\e^b,\quad b>2-\frac{\sqrt 5}{2}\end{equation} so that \begin{equation}\label{con444}|\beta|\frac{\rho_{2, \e}}{|\rho_\e|}e^{-\sqrt{\omega_3}\frac{|\rho_\e|}{\e}}\left(\frac{|\rho_\e|}{\e}\right)^{-\frac{N-1}{2}}=o\left(e^{-2\sqrt{\omega_2}\frac{\rho_{2, \e}}{\e}}\left(\frac{\rho_{2, \e}}{\e}\right)^{-\frac{N-1}{2}}\right)\end{equation} \begin{equation}\label{cond3}|\beta|\frac{\rho_{3, \e}}{|\rho_\e|}e^{-\sqrt{\omega_2}\frac{|\rho_\e|}{\e}}\left(\frac{|\rho_\e|}{\e}\right)^{-\frac{N-1}{2}}=o\left(e^{-2\sqrt{\omega_3}\frac{\rho_{3, \e}}{\e}}\left(\frac{\rho_{3, \e}}{\e}\right)^{-\frac{N-1}{2}}\right)\end{equation} and
  \begin{equation}\label{cond4}|\beta|\tau_\e=o\left(\e^{2}|\ln\e|\right)\end{equation}
 Then $\mathfrak f$ reduces to
 $$\mathfrak f(\rho_{2, \e}, \rho_{3, \e})= \mathfrak f (\rho_{2, \e})=-\partial_{11}^2\omega_2(0){\mathfrak  b}\e\rho_{2, \e}-2\mu_2{\mathfrak c} e^{-2\sqrt{\omega_2}\frac{\rho_{2, \e}}{\e}}\left(\frac{\rho_{2, \e}}{\e}\right)^{-\frac{N-1}{2}}$$ and again 
 $$\mathfrak f\left(\left(\frac{1}{\sqrt{\omega_2}}-\delta\right)\e|\ln\e|\right)<0,\quad \mathfrak f\left(\left(\frac{1}{\sqrt{\omega_2}}+\delta\right)\e|\ln\e|\right)>0$$ and hence \eqref{case3} has a zero at some $$\bar{\rho}_2:=\left(\frac{1}{\sqrt\omega_2}+o(1)\right)\e|\ln\e|.$$
Similarly, $\mathfrak g$ reduces to
$$\mathfrak g(\rho_{2, \e}, \rho_{3, \e})= \mathfrak g (\rho_{3, \e})=-\partial_{22}^2\omega_3(0)\bar{\mathfrak  b}\e\rho_{3, \e}-2\mu_3\bar{\mathfrak c} e^{-2\sqrt{\omega_3}\frac{\rho_{3, \e}}{\e}}\left(\frac{\rho_{3, \e}}{\e}\right)^{-\frac{N-1}{2}}$$ and again 
 $$\mathfrak g\left(\left(\frac{1}{\sqrt{\omega_3}}-\delta\right)\e|\ln\e|\right)<0,\quad \mathfrak g\left(\left(\frac{1}{\sqrt{\omega_3}}+\delta\right)\e|\ln\e|\right)>0$$ and hence \eqref{case4} has a zero at some $$\bar{\rho}_3:=\left(\frac{1}{\sqrt\omega_3}+o(1)\right)\e|\ln\e|.$$
 It only remains to prove that \eqref{con443} and \eqref{cond3} hold for any $\rho_{i, \e}\in\mathcal D_{i, \e}^3$.\\
 Indeed, we have that $$\frac{|\beta|\frac{\rho_{2, \e}}{|\rho_\e|}e^{-\sqrt{\omega_3}\frac{|\rho_\e|}{\e}}\left(\frac{|\rho_\e|}{\e}\right)^{-\frac{N-1}{2}}}{e^{-2\sqrt{\omega_2}\frac{\rho_{2, \e}}{\e}}\left(\frac{\rho_{2, \e}}{\e}\right)^{-\frac{N-1}{2}}}\lesssim \e^{b+\sqrt{1+\frac{\omega_3}{\omega_2}}-2-\delta}\lesssim \e^{b+\frac{\sqrt 5}{2} -2-\delta}\to 0$$ provided $b>2-\frac{\sqrt 5}{2}$ and $\delta$ sufficiently small. Moreover
 $$\frac{|\beta|\frac{\rho_{3, \e}}{|\rho_\e|}e^{-\sqrt{\omega_2}\frac{|\rho_\e|}{\e}}\left(\frac{|\rho_\e|}{\e}\right)^{-\frac{N-1}{2}}}{e^{-2\sqrt{\omega_3}\frac{\rho_{3, \e}}{\e}}\left(\frac{\rho_{3, \e}}{\e}\right)^{-\frac{N-1}{2}}}\lesssim \e^{b+\sqrt{1+\frac{\omega_2}{\omega_3}}-2-\hat\delta}\lesssim \e^{b+\frac{\sqrt 5}{2} -2-\delta}\to 0$$ again provided $b>2-\frac{\sqrt 5}{2}$ and $\delta$ sufficiently small.\\ At the end, we also get
 $$\frac{|\beta|\tau_\e}{\e^2} \lesssim \e^{b-2} \left(\e^{2}|\ln\e|+|\beta|\e^{\frac{\sqrt 5}{2}-\sigma}\right)\lesssim \e^{b-\sigma}+\e^{2b-2+\frac{\sqrt 5}{2}-\sigma}\rightarrow 0.$$  If, finally, $\omega_2=\omega_3=\omega_0$ then again $\mathfrak f$ and $\mathfrak g$ are defined as in \eqref{case3} and \eqref{case4}. Reason exactly as in the previous cases we obtain the result.
 \end{proof}

\begin{proof}[Theorem \ref{thm2} ]
If we have only two populations, namely the number of equations are two in \eqref{DS}, then we can find the error term as in Proposition \ref{nonlin} and $$\|(\Phi_1, \Phi_2)\|_X \lesssim \e^2\|\ln\e|^2.$$ Now, by lemmas \ref{prolinN} and \ref{proiezioni} we get that
$\mathfrak c_2=0$ if and only if $$\left(-\partial^2_{11}\omega_2(0)\mathfrak b \e \rho_{2, \e}-2\mu_2 \mathfrak c e^{-2\sqrt{\omega_2}\frac{\rho_{2, \e}}{\e}}\left(\frac{\rho_{2, \e}}{\e}\right)^{-\frac{N-1}{2}}\right)(1+o(1))=0.$$
The theorem follows since it is possible to show that a zero at $$\bar\rho_2:=\left(\frac{1}{\sqrt{\omega_2}}+o(1)\right)\e|\ln\e|$$ occurs
 by only assuming that $a_{21}<0$. The proof goes in the same line of \cite{PPVV}.\end{proof}

\section{Further results}\label{generalizzazione}
The same idea can be used to consider the following system of Gross-Pitaevskii type
$$ \left\{
\begin{aligned}
-  \Delta u_1 +V (x)u_1&=\mu_1 u_1^3+ u_1 \sum\limits_{j=2}^k \beta_{1j} u_j^2 &
 \hbox{in}\ \mathbb R^N,
 \\
-\e^2 \Delta u_j +W_j(x)u_j&=\mu_j u_j^3+u_j  \left(\beta_{j1} u_1^2+ \sum\limits_{i\not=j} \beta_{ji} u_i^2\right) & \hbox{in}\ \mathbb R^N,\ j=2,\dots,k.
\end{aligned}\right.
$$
In \cite{PPVV}  the case of a binary mixture is considered.\\ 
If we let again $k=3$ and we make the usual change of variables
we obtain the following system
\begin{equation}\label{GP} \left\{
\begin{aligned}
-  \Delta u_1+V (x)u_1&=\mu_1 u^3_1+ u_1 \sum\limits_{j=2}^3 \beta_{1j} u_j^2\left(\frac x \e\right) &
 \hbox{in}\ \mathbb R^N,
 \\
-\Delta u_j +W_j(\e x)u_j&=\mu_j u_j^3+u_j  \left(\beta_{j1} u_1^2(\e x)+ \sum\limits_{i\not=j\atop i\neq 1} \beta_{ji} u_i^2\right) & \hbox{in}\ \mathbb R^N,\ j=2,3.
\end{aligned}\right.
\end{equation}
We assume that $({\bf V}_1)$, $({\bf V}_2)$ and $({\bf W}_i)$ hold and 
we let $X$ as in \eqref{X}. \\ We look for a solution $(u_1, u_2, u_3 )$ of \eqref{GP} of the form
$$u_1(x)=\Upsilon(x)+\varphi_1(x)+\Phi_1(x),\quad u_i(x)= \Ui(x)+\varphi_i(x)+\Phi_i(x), i=2, 3$$
where $\Upsilon$ solves \eqref{eq:V1}, $\varphi_1=\beta_{12}\eta_2+\beta_{13}\eta_3$ where $\eta_i\in H^2_V(\R^N)$ are the unique  even solution of the equation
\begin{equation} \label{eq:V}
-\Delta\eta_i  +\left(V(x)-3\mu_{1}\Upsilon^{2}(x)\right)\eta_i 
= \Upsilon(x) (\Ui)^{2}\left(\frac{x}\e\right).
\end{equation}
Moreover, $\varphi_1$ satisfies  $\|\varphi_1\|_{W^{2,m}(\R^N)}\lesssim \e^{N\over m}$ and 
$ \|\varphi_1\|_{C^{1,1-\frac N m}(\R^N)}\lesssim \e^{N\over m}$ for any $m\geq2.$
Moreover, for $i=2, 3$ we let $\varphi_i \in H^2 (\R^N)$,  solution of the equation 
\begin{equation}\label{eq:Psi}
-\Delta\varphi_i   + \left(\omega_i -3\mu_{i} \left((\Umm)^2+(\Ump)^2)\right) \right)\varphi_i  = 2 \beta_{i1}\Upsilon(0)\varphi_1(0) \Ui(x).
\end{equation}
Moreover,   $\|\varphi_i \|_{H^2(\R^N)}\lesssim \e^{N\over2}$ and 
 there exist $ \gamma_i,\, R_{0} >0$ such that
\[
| \varphi_i (x)|\lesssim \e^{N/2}\left( e^{-\gamma_i |x+\frac{P_{i}}\e|}+e^{-\gamma_i |x-\frac{P_{i}}\e|}\right),\qquad \forall\, |x|\geq R_{0}.
\]
Finally
\begin{equation}\label{hatU}
 \Ui(x)=\Um\left(x-\frac{P_i}{\e}\right)+\Um\left(x+\frac{P_i}{\e}\right):=\Umm+\Ump
\end{equation}
and the concentration points satisfy 
\begin{equation}\label{P1}
P_2=\rho_{2, \e} P_0=\rho_{2, \e}(1, 0, \ldots, 0),\,\, \frac{\rho_{2, \e}}{\e}\to+\infty\, \mbox{as}\,\, \e\to0.
\end{equation}
\begin{equation}\label{P2}
P_3=\rho_{3, \e} \bar P_0=\rho_{3, \e}(0, 1, \ldots, 0),\,\, \frac{\rho_{3, \e}}{\e}\to+\infty\, \mbox{as}\,\, \e\to0.
\end{equation}
Moreover $$\frac{|P_2\pm P_3|}{\e}=\frac{\sqrt{\rho_{2, \e}^2+\rho_{3, \e}^2}}{\e}\to +\infty\, \mbox{as}\,\, \e\to0.
$$ We let $\mathcal Z_\e$ and $\mathcal Y_\e$ the solutions of \eqref{Ze} and \eqref{Ye} respectively.\\\\
In what follows, we consider, for simplicity, only the case in which $\omega_i\equiv \omega_0$.\\\\
In order to simplify the notation we let 
$$\Gamma:=\Upsilon+\varphi_1,\quad \Psi_i:=\Ui+\varphi_i.$$
Again we rewrite the problem in the following form $$\mathcal L(\Phi)=\mathcal E+\mathcal N(\Phi)$$ where now
$$\mathcal L_1:=-\Delta\Phi_1+V(x)\Phi_1-3\mu_1 \Gamma^2\Phi_1-\Phi_1\sum_{i=2}^3\beta_{1i}\Psi_i^2\left(\frac x \e\right)-2\sum_{i=2}^3 \Gamma \Psi\left(\frac x \e\right)\Phi_i\left(\frac x \e\right)$$
$$\begin{aligned}\mathcal L_i&:=-\Delta\Phi_i+W_i(\e x)\Phi_i-3\mu_i \Psi_i^2\Phi_i-\beta_{i1}\Gamma^2(\e x)\Phi_i-2\beta_{i1}\Psi_i \Gamma(\e x)\Phi_1(\e x)\\&-2\sum_{j\neq i\atop j\neq 1}\beta_{ij}\Psi_i\Psi_j \Phi_j-\sum_{j\neq i\atop j\neq 1}\beta_{ij}\Psi_j^2\Phi_i\end{aligned}$$
$$\mathcal E_1:=\Delta\Gamma-V(x)\Gamma+\mu_1\Gamma^3+\sum_{j=2}^3\beta_{1j}\Gamma\Psi_i^2\left(\frac x \e\right)$$
$$\begin{aligned}\mathcal E_i&:=\Delta\Psi_i-W_i(\e x)\Psi_i+\mu_i\Psi_i^3+\beta_{i1}\Gamma^2(\e x)\Psi_i-\sum_{j\neq i\atop j\neq 1}\beta_{ij}\Psi_i\Psi_j^2\\
\end{aligned}$$
$$\mathcal N_1:=3\mu_1\Gamma\Phi_1^2+\mu_1\Phi_1^3+\sum_{j=2}^3\beta_{1j}\Gamma\Phi_j^2\left(\frac x \e\right)+\sum_{j=2}^3\beta_{1j}\Phi_1\Phi_j^2\left(\frac x \e\right)+2\sum_{j=2}^3\beta_{1j}\Phi_1\Psi_i\Phi_i\left(\frac x \e\right)$$
$$\begin{aligned}\mathcal N_i&:=3\mu_i\Psi_i\Phi_i^2+\mu_i\Phi_i^3+\beta_{i1}\Psi_i\Phi_1^2(\e x)+\beta_{i1}\Phi_i\Phi_1^2(\e x)\\
&+2\beta_{i1}\Gamma(\e x)\Phi_1(\e x)\Phi_i+\sum_{j\neq i \atop j\neq 1}(\Psi_i+\Phi_i)\Phi_j^2+2\sum_{j\neq i \atop j\neq 1}\beta_{ij}\Psi_j\Phi_i\Phi_j\end{aligned}$$
and hence again, letting $\mathcal K$ and $\Pi$ as before, we get that the problem that we have to solve is 
$$\left\{\begin{aligned} \Pi\left\{\mathcal L(\Phi)-\mathcal E-\mathcal N(\Phi)\right\}&=0\\
\Pi^\bot\left\{\mathcal L(\Phi)-\mathcal E-\mathcal N(\Phi)\right\}&=0.\end{aligned}\right.$$
Arguing as before one can show the following result.
\begin{proposition}\label{nonlinGP}
There exits $\e_0>0$ such that for every $\e\in(0, \e_0)$ there exists a unique solution $(\Phi_1, \Phi_2, \Phi_3)\in\mathcal K^\bot $ of the system \eqref{DSnonlin1}. Furthermore, $$\|(\Phi_1, \Phi_2, \Phi_3)\|_X\lesssim \e^{2}|\ln\e|^2.$$ \end{proposition}
We only remark that here we do not need the smallness of $\beta_{23}$ and $\beta_{32}$ since the term $$ \sum_{j\neq i\atop j\neq 1}\beta_{ij}\Psi_j\Psi_j\Phi_j$$ can be absorbed in the linear part. Indeed, $$\int_{\mathbb R^N}\Psi_2\Psi_3\Phi_3\mathcal Z_\e\, dx \lesssim \|(\Uae)^2\Ube\|_{L^2(\mathbb R^N)}\|\Phi_3\|=o(1)$$ as $\e\to 0$.
Now we project the error term obtaining the following result.
\begin{lemma}\label{proiezioniGP}
It holds
\begin{equation}\label{pro1}\int_{\RN}\mathcal E_2 \mathcal Z_\e\, dx=\left[-\partial^2_{11}\omega_2(0)\mathfrak b \e \rho_{2, \e}-2\mu_2 \mathfrak c e^{-2\sqrt{\omega_0}\frac{\rho_{2, \e}}{\e}}\left(\frac{\rho_{2, \e}}{\e}\right)^{-\frac{N-1}{2}}\right](1+o(1))\end{equation} for some positive $\mathfrak b$, $\mathfrak c$ and
\begin{equation}\label{pro2}\int_{\RN}\mathcal E_3 \mathcal Y_\e\, dx=\left[-\partial_{22}^2\omega_3(0)\bar{\mathfrak  b}\e\rho_{3, \e}-2\mu_3\bar{ \mathfrak c }e^{-2\sqrt{\omega_0}\frac{\rho_{3, \e}}{\e}}\left(\frac{\rho_{3, \e}}{\e}\right)^{-\frac{N-1}{2}}\right](1+o(1))\end{equation} for some positive $\bar{\mathfrak b}$, $\bar{\mathfrak c}$.
\end{lemma}
\begin{proof}
We evaluate only the interaction term.
Indeed, by using Lemma \ref{PV}
$$\begin{aligned}
\int_{\mathbb R^N}\Uae(\Ube)^2\mathcal Z_\e\, dx &=\frac 12 \int_{\mathbb R^N}\partial_1 U^2(y)U^2(y+\zeta)\, dy(1+o(1))+\frac 12 \int_{\mathbb R^N}\partial_1 U^2(y)U^2(y-\xi)\, dy(1+o(1))\\
&-\frac 12 \int_{\mathbb R^N}\partial_1 U^2(y)U^2(y+\xi)\, dy(1+o(1))-\frac 12 \int_{\mathbb R^N}\partial_1 U^2(y)U^2(y-\zeta)\, dy\\
&=-2\mathfrak d \left\{\begin{aligned} &\frac{\rho_{2, \e}}{|\rho_\e|}e^{-2\sqrt{\omega_0}\frac{|\rho_\e|}{\e}}\left(\frac{|\rho_\e|}{\e}\right)^{-\frac 12}\, \quad &\mbox{if}\, N=2\\
&\frac{\rho_{2, \e}}{|\rho_\e|}e^{-2\sqrt{\omega_0}\frac{|\rho_\e|}{\e}}\left(\frac{|\rho_\e|}{\e}\right)^{-2}\ln\frac{|\rho_\e|}{\e}\, \quad &\mbox{if}\, N=3\end{aligned}\right.
\end{aligned}$$
and it is easy to see that $$e^{-2\sqrt{\omega_0}\frac{|\rho_\e|}{\e}}=o\left(e^{-2\sqrt{\omega_0}\frac{\rho_{2,\e}}{\e}}\right)$$ since 
$|\rho_\e|>\rho_{2,\e}$.

\end{proof}

\begin{proof}[Theorem \ref{thm3}]
It follows reason as in the proof of Theorem \ref{thm1} and Theorem \ref{thm2}.
\end{proof}

\begin{remark}
If we want to consider  in \eqref{DS} or in \eqref{GP} the case of $k\geq 3$ equations then we conjecture that the possible configuration of peaks  is the following 

\begin{equation}\label{P1}
P_i=\rho_{i, \e} P_{0, i}=\rho_{i, \e}\left(\cos\frac{(i-2)\pi}{k-1}, \sin\frac{(i-2)\pi}{k-1}, 0, \ldots, 0\right) \,\, i=2, \ldots k
\end{equation}
and
\begin{equation}\label{P2}
\rho_{i, \e} \in \left[\left( \alpha_{i}-\delta\right)\e|\ln\e|, \left(\alpha_i+\delta\right)\e|\ln\e|\right],\,\, \frac{\rho_{i, \e}}{\e}\to+\infty\, \mbox{as}\,\, \e\to0
\end{equation}
where $\alpha_i$ are suitable chosen depending on $\omega_i$. Then the profile of the solution that we look for is
$$u_1(x)=\Upsilon(x)+\varphi_1(x)+\Phi_1(x),\quad u_i(x)= \Ui(x)+\varphi_i(x)+\Phi_i(x), i=2, \ldots, k$$
where
\begin{equation}\label{hatU}
 \Ui(x)=\Um\left(x-\frac{P_i}{\e}\right)+\Um\left(x+\frac{P_i}{\e}\right):=\Umm+\Ump
\end{equation}
and $\varphi_1, \varphi_i$ with $i=2, \ldots, k$ are the natural extension of the corrections previously defined.
\end{remark}

\section{Appendix}\label{appendix}
First we recall the following result.
\begin{lemma}[Lemma 3.7 in \cite{acr}]\label{ACR}
Let $u, v : \mathbb{R}^N \to \mathbb{R}$ be two positive continuous radial functions such that
\[
 u(x) \sim |x|^a e^{-b|x|}, \quad v(x) \sim |x|^{a'} e^{-b'|x|} 
 \]
as $|x| \to \infty$, where $a, a' \in \mathbb{R}$, and $b, b' >0$. Let $\xi \in \mathbb{R}^N $ such that $|\xi| \to \infty$. We denote $u_{\xi}(x)=u(x-\xi)$. Then the following asymptotic estimates hold:
\begin{itemize}
\item[(i)] If $b < b'$,
\[ 
\int_{\mathbb{R}^N} u_{\xi} v\sim e^{-b|\xi|} |\xi|^{a}. 
\]
A similar expression holds if $b > b'$, by replacing $a$ and $b$ with $a'$ and $b'$. 
\item[(ii)] If $b=b'$, suppose that $a \geq a'$. Then:
\[ \int_{\mathbb{R}^N} u_{\xi} v\sim
\begin{cases}
e^{-b|\xi|} |\xi|^{a+a'+\frac{N+1}{2}} & \text{ if } a' > -\frac{N+1}{2},\\
e^{-b|\xi|} |\xi|^{a} \log |\xi| & \text{ if } a' = -\frac{N+1}{2},\\
e^{-b|\xi|} |\xi|^{a}& \text{ if } a' < -\frac{N+1}{2}.
\end{cases} \]
\end{itemize}
\end{lemma}

Let $U_{\lambda, \mu}$ be the solution of $-\Delta U_{\l, \mu}+\l U_{\l, \mu}=\mu U_{\l, \mu}^3$.\\\\ In \cite{PV} (Lemma A.2) it is shown the following result.
\begin{lemma}\label{PV}
Let $s, t\geq 1$ and consider the following integral 
$$\Theta_{s, t}(\zeta):=\int_{\RN} U^s_{\l, \mu}(x+\zeta)\partial_{x_1}U^t_{\l, \mu}(x)\, dx, \quad\zeta\in\RN.$$
\begin{itemize}
\item[(i)] If $s<t$ then $$\Theta_{s, t}(\zeta)\sim \mathfrak c s\frac{\zeta_1}{|\zeta|}e^{-s\sqrt{\l}|\zeta|}|\zeta|^{-s\frac{N-1}{2}},\qquad \mbox{as}\,\, |\zeta|\to+\infty;$$
\item[(ii)] If $s=t$ then
$$ \Theta_{s, t}(\zeta)\sim\left\{\begin{aligned}&\mathfrak c s\frac{\zeta_1}{|\zeta|}e^{-s\sqrt{\l}|\zeta|}|\zeta|^{-s(N-1)+\frac{N+1}{2}},\quad &\mbox{if}\,\, s<\frac{N+1}{N-1}\\
&\mathfrak c s\frac{\zeta_1}{|\zeta|}e^{-s\sqrt{\l}|\zeta|}|\zeta|^{-s\frac{N-1}{2}}\ln |\zeta|,\quad &\mbox{if}\,\, s=\frac{N+1}{N-1}\\
&\mathfrak c s\frac{\zeta_1}{|\zeta|}e^{-s\sqrt{\l}|\zeta|}|\zeta|^{-s\frac{N-1}{2}},\quad &\mbox{if}\,\, s>\frac{N+1}{N-1}\end{aligned}\right.$$
\end{itemize}
as $|\zeta|\to+\infty$ where $\mathfrak c$ is a positive constant. In the case $N=1$ the first option in (ii) holds, namely as $|\zeta|\to+\infty$
$$\Theta_{s, t}(\zeta)\sim\mathfrak c s\frac{\zeta_1}{|\zeta|}e^{-s\sqrt{\l}|\zeta|}|\zeta|.$$
\end{lemma}
\begin{remark}\label{ossdecay}
We remark that if $U_{\l_i, \mu_i}$ satisfies $-\Delta U_{\l_i, \mu_i}+\l_i U_{\l_i, \mu_i}=\mu_i U_{\l_i, \mu_i}^3$ and let $$\Theta^{i, j}_{s, t}(\zeta):=\int_{\RN}U^s_{\l_j, \mu_j}(x+\zeta)\partial_{x_\ell} U^t_{\l_i, \mu_i}(x)\, dx,\quad \zeta\in\RN,\,\, \ell=1, 2, 3.$$  Then, by using the above lemmas, we get that 
\begin{itemize}
\item[(i)]
if $s\sqrt{\l_j}< t\sqrt{\l_i}$ for $i\neq j$ 
$$\Theta^{i, j}_{s, t}(\zeta)\sim \mathfrak c s \frac{\zeta_\ell}{|\zeta|}e^{-s \sqrt{\lambda_j}|\zeta|}|\zeta|^{-s\frac{N-1}{2}},\qquad \mbox{as}\,\, |\zeta|\to+\infty.$$
\item[(ii)] if $s\sqrt{\l_j}= t\sqrt{\l_i}$ for $i\neq j$, suppose that $s\leq t$ then 
$$ \Theta^{i, j}_{s, t}(\zeta)\sim\left\{\begin{aligned}&\mathfrak c \frac{\zeta_\ell}{|\zeta|}e^{-s\sqrt{\l_j}|\zeta|}|\zeta|^{-\frac{N-1}{2}(s+t)+\frac{N+1}{2}},\quad &\mbox{if}\,\, t<\frac{N+1}{N-1}\\
&\mathfrak c \frac{\zeta_\ell}{|\zeta|}e^{-s\sqrt{\l_j}|\zeta|}|\zeta|^{-s\frac{N-1}{2}}\ln |\zeta|,\quad &\mbox{if}\,\, t=\frac{N+1}{N-1}\\
&\mathfrak c \frac{\zeta_\ell}{|\zeta|}e^{-s\sqrt{\l_j}|\zeta|}|\zeta|^{-s\frac{N-1}{2}},\quad &\mbox{if}\,\, t>\frac{N+1}{N-1}\end{aligned}\right.$$
\item[(iii)] if $s\sqrt{\l_j}> t\sqrt{\l_i}$ for $i\neq j$ then
$$\Theta^{i, j}_{s, t}(\zeta)=-\int_{\RN}U_{\l_i, \mu_i}^t(x+(-\zeta))\partial_{x_1}U_{\l_j, \mu_j}^s(x)\, dx \sim \mathfrak c t\frac{\zeta_\ell}{|\zeta|}e^{-t \sqrt{\lambda_i}|\zeta|}|\zeta|^{-t\frac{N-1}{2}},\quad \mbox{as}\,\, |\zeta|\to+\infty.$$
\end{itemize}
\end{remark}

\end{document}